\documentclass[12pt]{article}
\RequirePackage[colorlinks,citecolor=blue,urlcolor=blue]{hyperref}

\usepackage{latexsym,amssymb,amsmath,amsfonts,amsthm}
\usepackage{amsthm,amsmath}
\usepackage{float}
\usepackage{enumerate}
\usepackage{epic}
\usepackage[margin=0.75 in]{geometry}
\usepackage{bbm}
\usepackage{subfigure}
\usepackage{mathbbol}
\usepackage{graphicx}
\usepackage[toc,page]{appendix}
\usepackage{multirow}
\usepackage{amsfonts}
\usepackage[all]{xy}
\usepackage{caption}

\newtheorem{theorem}{Theorem}

\newtheorem{lemma}{Lemma}[section]

\newtheorem{prop}{Proposition}
\newtheorem{corollary}{Corollary}

\def\beq{ \begin{equation} }
\def\eeq{ \end{equation} }
\def\mn{\medskip\noindent}

\def\ep{\epsilon}

\def\square{\vcenter{\vbox{\hrule height .4pt
  \hbox{\vrule width .4pt height 5pt \kern 5pt
        \vrule width .4pt} \hrule height .4pt}}}

\def\ZZ{\mathbb{Z}}

\begin{document}

\title{Convergence of Stochastic Interacting Particle Systems in Probability under a Sobolev Norm}
\author{Jian-Guo Liu \\
Dept. of Physics and Dept. of Mathematics, Duke U., USA\\
\ \\
Yuan Zhang \\
Dept. of Mathematics, UCLA, USA}
\date{}
\maketitle

\abstract{In this paper, we consider particle systems with interaction and Brownian motion. We prove that when the initial data is from the sampling of Chorin's method, i.e., the initial vertices are on lattice points $hi\in \mathbb{R}^d$ with mass $\rho_0(hi) h^d$, where $\rho_0$ is some initial density function, then the regularized empirical measure of the interacting particle system converges in probability to the corresponding mean-field partial differential equation with initial density $\rho_0$, under the Sobolev norm of $L^\infty(L^2)\cap L^2(H^1)$. Our result is true for all those systems when the interacting function is bounded, Lipschitz continuous and satisfies certain regular condition. And if we further regularize the interacting particle system, it also holds for some of the most important systems of which the interacting functions are not. For systems with repulsive Coulomb interaction, this convergence holds globally on any interval $[0,t]$. And for systems with attractive Newton force as interacting function, we have convergence within the largest existence time of the regular solution of the corresponding Keller-Segel equation.}

\section{Introduction}

\mn In this paper we consider the $N-$particle system of many indistinguishable individuals interacting with each other following the same physical laws. To be specific, we consider $\{X_i(t)\}_{i=1}^N\in \mathbb{R}^d$ as the trajectories of the $N$ particles at time $t$. Suppose all particles have the same ``weight", with certain initial data $\{X_i(0)\}_{i=1}^N$, those trajectories following the stochastic differential equations as follows:
\beq
\label{SDE0}
X_i(t)=X_i(0)+\frac{1}{N}\sum_{j=1}^N F_0\big(X_i(s)-X_j(s) \big)ds+\sigma B_i(t)
\eeq
where $\{B_i(t)\}_{i=1}^N$ are independent standard $d-$dimensional Brownian motions. We show that, as $N\to\infty$ and under proper assumption of the initial data, the regularized empirical measure of the interacting particle system converges in probability to the solution of the corresponding partial differential equation (PDE) as follows, which is also called the mean-field equation: 
\beq
\label{PDE}
\left\{
\begin{aligned}
&\frac{\partial\rho}{\partial t}(x,t)=\frac{1}{2}\Delta\rho-\nabla\cdot(\rho F(x,t))\\
&F(x,t)=\int_{\mathbb{R}^d} F_0(x-y)\rho(y,t)dy.
\end{aligned}
\right.
\eeq
The interest on such convergence was raised from the study of propagation of chaos, which was originated by Kac \cite{Kac}. It is of interest since that to prove the propagation of chaos, one need to prove that the empirical measure of the particle system converges in law to the solution of the mean-field PDE with a proper initial condition. See the review by Sznitman \cite{Sznitman} for reference.

\mn Following this method, the propagation of chaos has been proved for different types of systems since the 1970s. McKean \cite{McKean} proved the propagation of chaos when the interacting function $F_0$ is smooth. He also conjectured that when $F_0(x)=\delta(x)$, the one dimensional mean-field equation is the Burgers equation. This conjecture was proved \cite{Calderoni,Gutkin,Sznitman2}. More cases when $F_0$ is no longer smooth has been studied. For $d=2$ and the interacting force given by $F_0(x)=-\nabla^{\perp}\Phi(x)$ where $\nabla^{\perp}=(\frac{\partial}{\partial x_2}, \frac{\partial}{\partial x_1})$ and $\Phi(x)=-\frac{1}{2}\ln|x|$, then the mean-field equation becomes the incompressible Navier-Stokes equation. When
$\sigma=0$, it is the incompressible Euler equation. The mean-field limits of this type of system have been studied in \cite{Marchioro} and \cite{Osada}, with or without cut-off parameters. And when $d=3$, the three dimensional Navier-Stokes equation and path-wise convergence rate with the stochastic vortex method have also been studied in a more recent work of \cite{Fontbona}. And more recently, for the system with Newton/ Coulomb interaction, i.e., when $F_0(x)=\pm\nabla \Phi(x)$, $\forall x\in \mathbb{R}^d-\{0\}$, where
$$
\Phi(x)=\left\{
\begin{aligned}
-\frac{1}{2\pi}\ln|x|&, \ \ d=2\\
\frac{C_d}{|x|^{d-2}}&, \ \  d\ge 3. 
\end{aligned}
\right.
$$
where $C_d=\frac{1}{d(d-2)\alpha_d}$ and $\alpha_d=\frac{\pi^{d/2}}{\Gamma(d/2+1)}$, the mean-field limit and propagation of chaos was proved by Liu and Yang, \cite{Liu1, Liu2, Liu3}. We refer readers to \cite{Bolley1,Bolley2,Fournier,Rezakhanlou1,Shiga} for more instances of the study of propagation of chaos. And we also refer to \cite{Craig} for recent progress on a blob method for the aggregation equation.  

\mn However, the mean-field limit results that the interacting particle system converges to the solution of the corresponding PDE, in the study of propagation of chaos, are usually obtained in a relatively weak sense, where the distance between two density functions are defined as Wasserstein distance. In this paper, we are, to our knowledge, for the first time to prove the convergence of the regularized empirical measure of such interacting particle system to the corresponding mean-field PDE under a stronger, Sobolev distance. Our result is generally true for all $F_0$ that is bounded, Lipschitz continuous and satisfies a regularity condition that will be specified in \eqref{Function Regular}. And when $F_0$ is the Newton/ Coulomb interaction, it is also true when the interacting particle system is further regularized. For the Coulomb interaction when there is a repulsive interaction, our result remains true on any interval $[0,t]$, while when $F_0$ is the gradient of Newton potential, since the system now has a attractive interacting force, we have convergence within the largest existence time of the regular solution of the corresponding Keller-Segel equation.

\mn To specify the interacting particle system we study in this paper precisely, we first need to determine the initial data. Majorly speaking, there are two ways to set up the initial data. On one hand, some previous researches like \cite{Goodman,Liu1,Marchioro,Osada} took the initial positions as independent identically distributed random variables with common density $\rho_0$. This approach is also known as the Monte Carlo sampling. However, this method is often inefficient in the computation. On the other hand, in \cite{Chorin}, where Chorin first introduced the vortex method in 1973, initial positions of the vertices are assumed to be on the lattice points $hi\in \mathbb{R}^2$ with a weight function $\rho_0(hi)h^2$ determined by the initial density. This way of sampling has been used in \cite{Long} and more recently, in \cite{Liu}, and it will be the initial condition we use in this paper. 

\mn To be specific, let $\rho_0(x), x\in \mathbb{R}^d$ be the initial density that satisfies the followings: 
\begin{itemize}
\item $\rho_0(x)$ is compact supported with a compact set $D\subset \mathbb{R}^d$. And $\rho_0(x)\ge 0$ for all $x\in D$.
\item $\int_D \rho_0(x) dx=1$.
\item $\rho_0$ is a Lipschitz continuous function with Lipschitz constant $L_{\rho_0}$. 
\item $\rho_0(x)\in H^k(\mathbb{R}^d)$ for some $k\ge \frac{3}{2}d+2$.
\end{itemize}
For each $h>0$. Let set $\Theta_h\subset \ZZ^d$ be defined as follows: 
\beq
\label{Index}
\Theta_h=\{\theta: \theta\in \ZZ^d, h\theta\in D\}. 
\eeq
For each $\theta\in \Theta_h$, let 
$$
C(\theta,h)=h\theta+\left[-\frac{h}{2},\frac{h}{2}\right]^d. 
$$
And it is east to see that $\{C(\theta,h)\}_{\theta\in \Theta_h}$ is a family of non-overlapping boxes and 
$$
D\subset \bigcup_{\theta\in \Theta_h} C(\theta,h).
$$
Let $N_h={\rm card}(\Theta_h)$, i.e., the number of elements in $\Theta_h$. Then by definition it is easy to check that 
\beq
\label{Distance1}
L_D=\int_D 1 dx\le  h^d N_h\le \int_{D_1} 1 dx=U_D
\eeq
for all $h<1$, where $D_1=\{x\in \mathbb{R}^d: \inf_{y\in D}|x-y|_\infty\le 1\}$. And as $h\to\infty$, we have 
$$
\lim_{h\to 0} h^d N_h=\int_D 1 dx. 
$$

\mn [Remark]: In this paper, we will use $\|\cdot\|$ for the $L^2$ norm of a function or vector valued function. I.e., for any $f(x)$, $x\in \mathbb{R}^d$ 
$$
\|f(x)\|^2= \int_{\mathbb{R}^d} |f(x)|^2 dx 
$$
and similar for the $L^p$ norm $\|\cdot\|_p$. And we will use $|\cdot|$ for the $L^2$ norm of a vector. I.e., for any vector $x=(x_1,x_2,\cdots, x_n)\in \mathbb{R}^d$, 
$$
|x|^2=x_1^2+x_2^2+\cdots+x_d^2
$$
and similar for the $L^p$ norm $|\cdot|_p$.

\mn For each $h>0$, since $\Theta_h$ is finite, we can have all its elements ordered under a natural ordering:
$$
\Theta_h=\{\theta_{1,h},\theta_{2,h},\cdots, \theta_{N_h,h}\} 
$$
and have the initial point of the $i$th particle to be $\theta_{i,h}$, $i=1,2,\cdots, N_h$. 

\mn With the initial data specified, we now formally introduce the stochastic interacting particle system in our paper: 

\mn When $F_0$ is bounded and Lipschitz continuous, let $\{X_{h,i}(t)\}_{i=1}^{N_h}$ be the {\bf interacting particle system} determined by the following system of SDE:
\beq
\label{particle system}
X_{h,i}(t)=\theta_{i,h}h+\int_0^t \left(\sum_{j=1}^{N_h} F_0\big(X_{h,i}(s)-X_{h,j}(s)\big)\rho_0(\theta_{j,h}) h^d \right)ds+B_i(t), \ i=1,2,\cdots, N_h. 
\eeq
Noting that $F_0$ is bounded and Lipschitz continuous, the SDE in \eqref{particle system} always has a unique strong solution. 

\mn When $F_0$ is not bounded and Lipschitz continuous, in order to have a SDE with a unique strong solution, we need to define $\{X_{h,i,\delta_h}(t)\}_{i=1}^{N_h}$ to be the  {\bf regularized interacting particle system} as follows:
\beq
\label{regularized particle system}
X_{h,i,\delta_h}(t)=\theta_{i,h}h+\int_0^t \left(\sum_{j=1}^{N_h} F_{0,\delta_h}\big(X_{h,i,\delta_h}(s)-X_{h,j,\delta_h}(s) \big)\rho_0(\theta_{j,h}) h^d \right)ds+B_i(t).
\eeq
where 
$$
F_{0,\delta_h}=F_0*\psi_{\delta_h}, \ \ \psi_{\delta_h}(x)=\delta_h^{-d} \psi(\delta_h x)
$$
and
$$
\psi(x)=\left\{
\begin{aligned}
&C(1+\cos \pi|x|)^{d+2}, \ \ &|x|\le 1\\
&0, \ \ &|x|>1
\end{aligned}
\right.
$$
with $C$ such that $\int_{\mathbb{R}^d} \psi(x) dx=1$. Here $\delta_h$ is some number goes to 0 as $h\to0$. It will be specified later in \eqref{Delta h}. 
 
\mn With the (regularized) interacting particle system determined, we define the regularized empirical measure as follows: consider a function $\varphi(x)\in C_0^\infty(\mathbb{R}^d)$ of which the support is $\{|x|_\infty\le 1/2\}$ such that
\begin{itemize}
\item $\varphi(x)\ge 0$.
\item $\int_{|x|_\infty\le 1} \varphi(x)=1$. 
\end{itemize}
 And for $\epsilon_h=h^{q_0}$ where $q_0$ is to be specified later in Theorem 1, let 
$$
\varphi_{\ep_h}(x)=\frac{1}{\ep_h^d} \varphi\left( \frac{x}{\ep_h}\right). 
$$ 
Then for the interacting particle system when $F_0$ is bounded and Lipschitz continuous, the regularized empirical measure of $\{X_{h,i}(t)\}_{i=1}^{N_h}$ is defined as
\beq
\label{empirical1}
\rho_{h}(x,t)=\sum_{i=1}^{N_h} h^d \rho_0(\theta_{i,h}) \varphi_{\ep_h}\left(x-X_{h,i}(t)\right). 
\eeq
And the regularized empirical measure of $\{X_{h,i,\delta_h}(t)\}_{i=1}^{N_h}$ is similarly defined as 
\beq
\label{empirical1 regular}
\rho_{h,\delta_h}(x,t)=\sum_{i=1}^{N_h} h^d \rho_0(\theta_{i,h}) \varphi_{\ep_h}\left(x-X_{h,i,\delta_h}(t)\right). 
\eeq
The use of the such regularized empirical measure as above is important in computation and the regularized kernel $\varphi$ is known as a blob function in the vortex method. Pioneered by Chorin in 1973 \cite{Chorin}, the random vertex blob method is one of the most successful computational methods for fluid dynamics and other related fields. The success of the method is exemplified by the accurate computation of flow past a cylinder at the Reynolds numbers up to $9500$ in the 1990s \cite{KL}. The convergence analysis for the random vortex method for the Navier-Stokes equation is given by \cite{Goodman, Long, Marchioro} in the 1980s. We refer  to the book \cite{CK} for theoretical and practical use of vortex methods, refer to Goodman \cite{Goodman} and Long \cite{Long} for the convergence analysis of the random vortex method to the Navier-Stokes equation. We also hoped that the estimation in this paper can be adapted to do numerical analysis.

\mn With the regularized empirical measure defined, we need to add one more regularity condition on $F_0$ which assumes the existence of a constant $U_{F}<\infty$ such that 
\beq
\label{Function Regular}
\|F_0\|_{L^1\cap H^{2d+2}(\mathbb{R}^d)}\le U_F. 
\eeq
Then we can have our main result of this paper, which sates that the regularized empirical measure of the interacting particle system converges to the solution of the PDE. It is presented in the following theorem: 

\begin{theorem}
Suppose $F_0(x)$ is a bounded and Lipschitz continuous in $\mathbb{R}^d$, where the Lipschitz constant of $F_0(x)$ is given by $L_F$. And suppose that $F_0$ satisfies condition \eqref{Function Regular}. Let $\{X_{h,i}(s)\}_{i=1}^{N_h}$ be the interacting particle system defined in \eqref{particle system} and $\rho_{h}$ be the constructed regularized empirical measure (\ref{empirical1}) with regularized parameter $\ep_h=h^{1/6d}$. Let $\rho$ be the solution of the corresponding mean-field equation (\ref{PDE}) with initial density $\rho_0$. 
Then, there is a positive function $c(t)$ (will be specified in (\ref{FunC})) dependent only on $t$, $\varphi$, $L_F$, $U_F$ and $\rho_0$, and a $h_0>0$, such that
\beq
\label{MainRe}
\begin{aligned}
P\Big(&\sup_{s\in [0,t]}\big (\| (\rho-\rho_{h})(\cdot,s)\|^2+\int_0^ s\|\nabla(\rho-\rho_{h})(\cdot,q)\|^2\, dq \big ) <c(t) h^{1/12d}  \Big)\ge 1-c(t)\, h^{1/12d}
\end{aligned}
\eeq
for all $0<h\le h_0$. 
\end{theorem}

\subsection{Outline of the Proof}
Most of the rest of the paper will be devoted to the proof of Theorem 1. The {\bf key idea} of the proof is to introduce an intermediate system of {\bf self-consistent process} $\hat X_{h,i}(t)$ defined by 
\beq
\label{self-consistent}
\hat X_{h,i}(t)=\theta_{i,h}h+\int_0^t F\big(\hat X_{h,i}(s),s\big) ds+B_i(t)
\eeq
where $\{B_i(t)\}_{i=1}^{N_h}$ are the same family of standard Brownian motions as in \eqref{particle system}, and $F(x,t)$ is defined in \eqref{PDE}. The first thing we note is that $F(x,t)$ is a bounded and Lipschitz function against $x$ with Lipschitz constant less than or equal to $L_F$. First for any $x$ and $t$,  
$$
|F(x,t)|\le \int_{\mathbb{R}^d} |F_0(x-y)|\rho(y,t) dy.
$$
Noting that $\rho$ is a probability density function on $\mathbb{R}^d$, we have
$$
|F(x,t)|\le \sup_{x\in \mathbb{R}^d} |F_0(x)|<\infty.
$$
And similarly, for any $t\ge 0$ and $x_1,x_2\in \mathbb{R}^d$, we have 
\begin{align*}
|F(x_1,t)-F(x_2,t)|&\le \int_{\mathbb{R}^d} |F_0(x_1-y)-F_0(x_2-y)|\rho(y,t) dy\\
&\le \sup_{y\in \mathbb{R}^d} |F_0(x_1-y)-F_0(x_2-y)|\\
&\le |x_2-x_1| L_F.
\end{align*}
Thus, $\{\hat X_{h,i}(t)\}_{i=1}^{N_h}$ is a family of independent strong solutions the same SDE with the same initial values as the interacting particle system. Then consider the similar regularized empirical measure:
\beq
\label{empirical2}
\hat \rho_{h}(x,t)=\sum_{i=1}^{N_h} h^d \rho_0(\theta_{i,h}) \varphi_{\ep_h}\left(x-\hat X_{h,i}(t)\right). 
\eeq
If we can estimate both the distances between $\rho_{h}(x,t)$ and $\hat \rho_{h}(x,t)$, and the distance between $\hat \rho_{h}(x,t)$ and $\rho(x,t)$, we will be able to prove Theorem 1. 

\vspace{0.1 in}

\mn{\bf I. Control the Distance Between $\rho_{h}(x,t)$ and $\hat \rho_{h}(x,t)$}

\mn To estimate the distance between the regularized empirical measure constructed from the interacting particle system and that constructed from the self-consistent process, we use a recent proved result by Huang and Liu \cite{Liu} that estimates the $l_h^p$ norm of $\hat X_{h,i}(t)-X_{h,i}(t)$. For any vector $\vec x=(x_1,x_2,\cdots, x_{N_h})$ and $p\ge1$, its $l_h^p$ norm is defined as 
\beq
\label{norm1}
\big|\vec x\big|_{l_h^p}=\left( h^d\sum_{i=1}^{N_h} |x_i|^p\right)^{\frac{1}{p}}. 
\eeq
According to Theorem 6.1 in \cite{Liu}, there exist a $p>1$ such that for all $0<h\le h_0$ with $h_0$ sufficiently small, there exists two positive constants $C$ and $C'$ depending on $t$, $p$, $d$, $U_F$, $\rho_0$ and the diameter of $D$. The following estimate holds true:
\beq
\label{distance 1}
P\left(\max_{0\le s\le t}\big| \hat X_{h,i}(s)-X_{h,i}(s)\big|_{l_h^p}< \Lambda h |\ln h| \right)\ge 1- h^{C\Lambda |\ln h|}
\eeq
for all $\Lambda\ge C'$. Then under the high probability event 
$$
E_h=\left\{ \max_{0\le s\le t}\big| \hat X_{h,i}(s)-X_{h,i}(s)\big|_{l_h^p}< C' h |\ln h| \right\},
$$
 we can use the $l_h^p$ norm of $ \hat X_{h,i}(t)-X_{h,i}(t)$ to estimate the distance between the empirical measures $\rho_{h}(x,t)$ and $\hat \rho_{h}(x,t)$. Details of this part can be found in Section 2. 

\vspace{0.1 in}

\mn{\bf II. Control the Distance Between $\rho(x,t)$ and $\hat \rho_{h}(x,t)$}

\mn In this second step we estimate the distance between the empirical measure $\hat \rho_{h}(x,t)$ constructed from the self-consistent process and the solution of the PDE. To estimate this distance, we have a theorem as follows:

\begin{theorem}
Let  $\{\hat X_{h,i}(t)\}_{i=1}^{N_h}$ be the self-consistent system and $\hat \rho_{h}(x,t)$ be the constructed regularized empirical measure with regularized parameter $\ep_h=h^{1/6d}$. Let $\rho$ be the solution of the corresponding mead field equation (\ref{PDE}) with initial density $\rho_0$. 
Then, there is a positive function $c_1(t), t> 0$ (will be specified in (\ref{FunC1})) dependent only on $t$, $\varphi$, and $\rho_0$, such that 
\beq
\label{MainRe2}
\begin{aligned}
P\Big(&\sup_{s\in [0,t]}\big (\| (\rho-\hat\rho_{h})(\cdot,s)\|^2+\int_0^ s\|\nabla(\rho-\hat\rho_{h})(\cdot,q)\|^2\, dq \big)\le c_1(t) \, h^{1/12d}   \Big)\ge 1-c_1(t)\, h^{1/12d}.
\end{aligned}
\eeq
where $C_0=2d L_F$ and $L_F$ is the Lipschitz constant of $F_0$.
\end{theorem}

\mn The proof of Theorem 2 is similar to the one reported recently in the authors' conference note \cite{Liu4}, where some preliminary work of this paper is reported with a much simplified system with only drift and diffusion but no interactions. In that case, the mean-field PDE is the Fokker-Planck equation and there are no mass function on each data points in the empirical measure. Here we generalized the proof and make it adapted to the new definition of empirical measure in this paper and to the self-consistent system. To prove this theorem, we take the following steps: 
\begin{enumerate}[(1)]
\item First we use Ito's formula to decompose the $L^\infty(L^2)\cap L^2(H^1)$ norm of the difference between $\rho(x,t)$ and $\hat \rho_{h}(x,t)$. Here we will have a term that is from the free energy estimation of PDE, a term of initial error, a term of truncation error, and a term of martingale error. Details can be found in Section 3.
\item Second, we prove a proposition on the separation of the self-consistent system which shows that for the self-consistent system $\{\hat X_{h,i}(t)\}_{i=1}^{N_h}$, there cannot be too many particles stay too close with each other. To prove this separation problem, we use Girsanov Theorem to reduce it to a separation problem of standard Brownian motions. The proof of the Brownian motion case is technical, where cases for $d=2$, $d=2$ or $d\ge3$ will be proved differently. Details can be found in Section 4.   
\item Then we estimate the term of truncation error.  We are able to use the result we proved in the proposition of separation and the fact that $\varphi_{\ep_h}$ is supported on $\{x: |x|_\infty<\ep_h\}$ to bounded the truncation error under a high probability event. Details can be found in Section 5. 
\item Since the empirical measure is rescaled by $h^d$, we can use standard stochastic differential equation argument to estimate the martingale errors in the estimation. Details can be found in Section 6. 
\item Note that $\varphi\in C_0^\infty$ and that the initial density $\rho_0$ is Lipschitz continuous. We can estimate the initial error using standard calculations. Details can be found in Section 7. 
\item After we have estimated the initial, truncation and the martingale errors, we can use Gronwall's inequality to estimate the distance between the empirical measure $\hat \rho_{h}(x,t)$ and the solution of the PDE and finish the proof.  Details can be found in Section 8.
\end{enumerate}

\mn Combining Part I and Part II, we finish the proof of Theorem 1.

\subsection{Newton and Coulomb Interactions}
With Theorem 1 holds true, when the interacting function $F_0$ is not bounded and Lipschitz continuous, this convergence result may remain hold. The intuition behind this generalization is that, though $F_0$ itself is not bounded and Lipschitz continuous, the function $F$ defined in \eqref{PDE} may still be bounded and Lipschitz continuous (in a certain interval). Thus the SDE of self-consistent system in \eqref{self-consistent} is still well defined and has a unique strong solution. Note that the proof of Theorem 2 depends only on the fact that $F$ rather than $F_0$ is bounded and Lipschitz continuous against $x$. We are still able to estimate the distance between the regularized empirical measure $\hat\rho_{h}(x,t)$ of the self-consistent system, and the solution of the PDE.

\mn Thus, to show the convergence result in Theorem 1, it suffices to estimate the distance between the regularized empirical measure $\hat\rho_{h}(x,t)$ of the self-consistent system and the regularized empirical measure $\rho_{h,\delta_h}(x,t)$ of the regularized interacting particle system. Fortunately, the results recently proved in \cite{Liu} give us exactly the same estimation as in \eqref{distance 1}, between of the $l_h^p$ distance between $\hat X_{h,i}(s)$ and $X_{h,i,\delta_h}(s)$, when the function $F_0$ is Coulomb or Newton Interactions. Thus exactly the same argument as in Section 2 will finish the proof for those systems. 

\mn{\bf Newton Interaction.} In this case, the aggregation function is given by $F_0(x)=\nabla \Phi(x)$, $\forall x\in \mathbb{R}^d-\{0\}$, where
$$
\Phi(x)=\left\{
\begin{aligned}
-\frac{1}{2\pi}\ln|x|&, \ \ d=2\\
\frac{C_d}{|x|^{d-2}}&, \ \  d\ge 3. 
\end{aligned}
\right.
$$
And the mean-field PDE is the Keller-Segel equation. Noting that $\rho_0(x)\in H^k(\mathbb{R}^d)$ for some $k\ge \frac{3}{2}d+1$, this implies the existence of the unique local solution to the Keller-Segel equation with the follow regularities
\beq
\label{regularity1}
\|\rho\|_{L^\infty(0,T,H^k(\mathbb{R}^d))}\le C(\|\rho_0\|_{H^k(\mathbb{R}^d)})
\eeq
and
\beq
\label{regularity2}
\|\partial_t\rho\|_{L^\infty(0,T,H^{k-2}(\mathbb{R}^d))}\le C(\|\rho_0\|_{H^k(\mathbb{R}^d)})
\eeq
where $T>0$ depends only on $\|\rho_0\|_{H^k(\mathbb{R}^d)}$. Denote $T_{\max}$ to be the largest existence time such that \eqref{regularity1} and \eqref{regularity2} is valid. According to Sobolev imbedding theorem, one has $\rho(x,t)\in C^{k-d/2-1}$ for any $t\in [0, T]$. And for 
$$
F(x,t)=\int_{\mathbb{R}^d}F_0(x-y)\rho(y,t)dy
$$
using the Sobolev imbedding theorem again gives us 
$$
\|F\|_{L^\infty(0,T,W^{k-d/2-2,\infty}(\mathbb{R}^d))}\le C \|F\|_{L^\infty(0,T,H^{k+1}(\mathbb{R}^d))}\le C(\|\rho_0\|_{H^k(\mathbb{R}^d)})
$$
and 
$$
\|\partial_t F\|_{L^\infty(0,T,W^{k-d/2-2,\infty}(\mathbb{R}^d))}\le C(\|\rho_0\|_{H^k(\mathbb{R}^d)}).
$$
Thus for any $T<T_{\max}$, $F(x,s)$ is a bounded and Lipschitz continuous on $\mathbb{R}^d\times[0,T]$, with Lipschitz constants uniformly bounded.  Thus, Theorem 2 in this paper holds for the regularized empirical measure $\hat\rho_{h}(x,t)$ of the self-consistent system, and the solution of the corresponding Keller-Segel equation. 

\mn With the distance between the self-consistent system and the mean-field PDE estimated, let 
\beq
\label{Delta h}
\delta_h=h^\kappa
\eeq
where $\kappa\in (1/2,1)$. Then according to Theorem 1.1 in \cite{Liu}, we have for $p>d/(1-\kappa)$, and $h$ sufficiently small, there exists two positive constants $C$ and $C'$ depending on $T_{\max}$, $p$, $d$ and $\rho_0$ and the diameter of $D$. The following estimate holds true:
$$
P\left(\max_{0\le s\le t}\big| \hat X_{h,i}(s)-X_{h,i}(s)\big|_{l_h^p}< \Lambda h |\ln h| \right)\ge 1- h^{C\Lambda |\ln h|}
$$
for all $\Lambda\ge C'$. Noting that the inequality above has the same form as \eqref{distance 1}, then the argument in Section 2 gives the estimation between $\hat X_{h,i}(s)$ and $X_{h,i,\delta_h}(s)$ and gives us the following corollary on the convergence of system with Newton interaction:
\begin{corollary}
For any $t<T_{\max}$, suppose $F_0(x)$ is given by the Newton Interaction. Let $\{X_{h,i,\delta_h}(s)\}_{i=1}^{N_h}$ be the regularized interacting particle system defined in \eqref{regularized particle system} with $\delta_h$ defined in \eqref{Delta h}, and $\rho_{h,\delta_h}$ be the constructed regularized empirical measure (\ref{empirical1 regular}) with regularized parameter $\ep_h=h^{1/6d}$. Let $\rho$ be the solution of the corresponding Keller-Segel equation with initial density $\rho_0$. 
Then, there is a positive function $c(t)$ (will be specified in (\ref{FunC})) dependent only on $t$, $\varphi$, $L_F$ and $||\rho_0||$, and a $h_0>0$, such that
\beq
\begin{aligned}
P\Big(&\sup_{s\in [0,t]}\big (\| (\rho-\rho_{h,\delta_h})(\cdot,s)\|^2+\int_0^ s\|\nabla(\rho-\rho_{h,\delta_h})(\cdot,q)\|^2\, dq \big ) <c(t) h^{1/12d}  \Big)\ge 1-c(t)\, h^{1/12d}
\end{aligned}
\eeq
for all $0<h\le h_0$. 
\end{corollary}

\mn {\bf Coulomb Interaction.}  In this case, the interaction function is given by $F_0(x)=-\nabla \Phi(x)$, $\forall x\in \mathbb{R}^d-\{0\}$, where
$$
\Phi(x)=\left\{
\begin{aligned}
-\frac{1}{2\pi}\ln|x|&, \ \ d=2\\
\frac{C_d}{|x|^{d-2}}&, \ \  d\ge 3. 
\end{aligned}
\right.
$$
And the mean-field PDE is the drift-diffusion equation. Thus again let $T_{\max}$ be the same largest existence time of a regular solution. According to \cite{Liu3}, $T_{\max}=\infty$ So again using Sobolev embedding theorem on 
$$
F(x,t)=\int_{\mathbb{R}^d}F_0(x-y)\rho(y,t)dx
$$
we have that for any $t>0$, $F(x,s)$ is bounded and Lipschitz continuous on $\mathbb{R}^d\times[0,t]$, with Lipschitz constants uniformly bounded. Thus, Theorem 2 in this paper holds for the regularized empirical measure $\hat\rho_{h}(x,t)$ of the self-consistent system, and the solution of the corresponding drift-diffusion equation.

\mn Moreover, according to exactly the same argument, see Remark 1.1 in \cite{Liu}, let $\delta_h$ be the same as defined in \eqref{Delta h}, there is a $p>d/(1-\kappa)$, and $h$ sufficiently small, there exists two positive constants $C$ and $C'$ depending on $T_{\max}$, $p$, $d$ and $\rho_0$ and the diameter of $D$. The following estimate holds true:
$$
P\left(\max_{0\le s\le t}\big| \hat X_{h,i}(s)-X_{h,i}(s)\big|_{l_h^p}< \Lambda h |\ln h| \right)\ge 1- h^{C\Lambda |\ln h|}
$$
for all $\Lambda\ge C'$. Thus we have the following corollary on the convergence of system with Coulomb interaction:
\begin{corollary}
For any $t>0$, suppose $F_0(x,s)$ is given by the Coulomb Interaction. Let $\{X_{h,i,\delta_h}(s)\}_{i=1}^{N_h}$ be the regularized interacting particle system defined in \eqref{regularized particle system} with $\delta_h$ defined in \eqref{Delta h}, and $\rho_{h,\delta_h}$ be the constructed regularized empirical measure (\ref{empirical1 regular}) with regularized parameter $\ep_h=h^{1/6d}$. Let $\rho$ be the solution of the corresponding drift-diffusion equation with initial density $\rho_0$. 
Then, there is a positive function $c(t)$ (will be specified in (\ref{FunC})) dependent only on $t$, $\varphi$, $L_F$ and $||\rho_0||$, and a $h_0>0$, such that
\beq
\begin{aligned}
P\Big(&\sup_{s\in [0,t]}\big (\| (\rho-\rho_{h,\delta_h})(\cdot,s)\|^2+\int_0^ s\|\nabla(\rho-\rho_{h,\delta_h})(\cdot,q)\|^2\, dq \big ) <c(t) h^{1/12d}  \Big)\ge 1-c(t)\, h^{1/12d}
\end{aligned}
\eeq
for all $0<h\le h_0$. 
\end{corollary}

\section{The $L^\infty(L^2)\cap L^2(H^1)$ Distance between $\rho_h(x,t)$ and $\hat \rho_{h}(x,t)$}

According to Theorem 6.1 in \cite{Liu}, it has been proved that let $\{ X_{h,i}(t)\}_{i=1}^{N_h}$ and $\{\hat X_{h,i}(t)\}_{i=1}^{N_h}$ be the original interacting particle system and the self-consistent system with initial values of $\{\theta_{i,h}h\}_{i=1}^{N_h}$, which are specified in \eqref{particle system} and \eqref{self-consistent}, then  there exist a $p>1$ such that for all $0<h\le h_0$ with $h_0$ sufficiently small, there exists two positive constants $C$ and $C'$ depending on $t$, $p$, $d$ and $\rho_0$ and the diameter of $D$. The following estimate holds true for all $\Lambda >C'$:
$$
P\left(\max_{0\le s\le t}\big| \hat X_{h,i}(s)-X_{h,i}(s)\big|_{l_h^p}< \Lambda h |\ln h| \right)\ge 1- h^{C\Lambda |\ln h|}
$$
where the $l_h^p$ norm is defined in \eqref{norm1}. Then under the high probability event 
$$
E_h=\left\{ \max_{0\le s\le t}\big| \hat X_{h,i}(s)-X_{h,i}(s)\big|_{l_h^p}< C' h |\ln h| \right\},
$$
we will use the estimation of the distance between $\{ X_{h,i}(s)\}_{i=1}^{N_h}$ and $\{\hat X_{h,i}(s)\}_{i=1}^{N_h}$ to estimate the distance between the two empirical measures constructed from them. We have the theorem as follows: 

\begin{theorem}
Let $\{X_{h,i}(t)\}_{i=1}^{N_h}$ be the interacting particle system defined in \eqref{particle system} and $\{\hat X_{h,i}(t)\}_{i=1}^{N_h}$ be the self-consistent system defined in \eqref{self-consistent}. $\rho_{h}(x,t)$ and $\hat \rho_{h}(x,t)$ be the constructed regularized empirical measure respectively, with regularized parameter $\ep_h=h^{1/6d}$. Then, there is a positive function $c_0(t), t> 0$ (will be specified in (\ref{C0})) dependent only on $t$, $\varphi$, $U_F$, $L_F$, and $\rho_0$, such that 
\beq
\label{MainRe3}
\begin{aligned}
P\Big(&\sup_{s\in [0,t]}\big (\| (\rho_h-\hat\rho_{h})(\cdot,s)\|^2+\int_0^ s\|\nabla(\rho_h-\hat\rho_{h})(\cdot,q)\|^2\, dq \big )< c_0(t) \, h^{1/12d}  \Big)\ge 1-c_0(t)\, h^{1/12d}.
\end{aligned}
\eeq
\end{theorem}

\begin{proof}
Recall that 
$$
\rho_{h}(x,s)=\sum_{i=1}^{N_h} h^d \rho_0(\theta_{i,h}) \varphi_{\ep_h}\Big(x-X_{h,i}(s)\Big)
$$
and
$$
\hat \rho_{h}(x,s)=\sum_{i=1}^{N_h} h^d \rho_0(\theta_{i,h}) \varphi_{\ep_h}\left(x-\hat X_{h,i}(s)\right). 
$$
Then for any $s\in [0,t]$, the $L^2$ norm of the difference is given by  
\beq
\label{L21}
\begin{aligned}
\left\|\rho_{h}(x,s)-\hat \rho_{h}(x,s)\right\| &\le \sum_{i=1}^{N_h} h^d \rho_0(\theta_{i,h}) \left\| \varphi_{\ep_h}\Big(x-X_{h,i}(s)\Big)-\varphi_{\ep_h}\left(x-\hat X_{h,i}(s)\right) \right\|\\
&=\sum_{i=1}^{N_h} h^d \ep_{h}^{-d} \rho_0(\theta_{i,h}) \left\| \varphi\left(\frac{x-X_{h,i}(s)}{\ep_h}\right)-\varphi\left(\frac{x-\hat X_{h,i}(s)}{\ep_h}\right) \right\|.
\end{aligned}
\eeq
Note that since $\varphi\in C_0^\infty$, for any $i$, according to mid-value theorem
$$
\left| \varphi\left(\frac{x-X_{h,i}(s)}{\ep_h}\right)-\varphi\left(\frac{x-\hat X_{h,i}(s)}{\ep_h}\right) \right|\le c_d\left| \frac{X_{h,i}(s)-\hat X_{h,i}(s)}{\ep_h}\right|\max_{k=1,\cdots,d}\left\|\frac{\partial\varphi}{\partial x_k}\right\|_\infty
$$
where $c_d$ is some constant that depends only on $d$, and that 
$$
\left| \varphi\left(\frac{x-X_{h,i}(s)}{\ep_h}\right)-\varphi\left(\frac{x-\hat X_{h,i}(s)}{\ep_h}\right) \right|\equiv 0
$$
for all $x\notin U_i=\{y: \big|y-X_{h,i}\big|\le \ep_h/2\}\cup \{y: \big|y-\hat X_{h,i}\big|\le \ep_h/2\}$. Thus
\beq
\label{L22}
\begin{aligned}
 \left\| \varphi\left(\frac{x-X_{h,i}(s)}{\ep_h}\right)-\varphi\left(\frac{x-\hat X_{h,i}(s)}{\ep_h}\right) \right\|&\le c_d\left| \frac{X_{h,i}(s)-\hat X_{h,i}(s)}{\ep_h}\right|\max_{k=1,\cdots,d}\left\|\frac{\partial\varphi}{\partial x_k}\right\|_\infty \sqrt{\int_{U_i} dx}\\
 &\le 2c_d\left| \frac{X_{h,i}(s)-\hat X_{h,i}(s)}{\ep_h}\right|\max_{k=1,\cdots,d}\left\|\frac{\partial\varphi}{\partial x_k}\right\|_\infty \ep_h^{d/2}.
 \end{aligned}
\eeq
Plugging \eqref{L22} in \eqref{L21} and noting that $\rho_0$ is Lipschitz continuous and thus bounded, we have 
\beq
\label{L23}
\begin{aligned}
\left\|\rho_{h}(x,s)-\hat \rho_{h}(x,s)\right\| &\le 2c_d\|\rho_0\|_\infty \max_{k=1,\cdots,d}\left\|\frac{\partial\varphi}{\partial x_k}\right\|_\infty \ep_h^{-d/2-1}\big| \hat X_{h,i}(s)-X_{h,i}(s)\big|_{l_h^1}. 
\end{aligned}
\eeq
Noting that according to Jensen's inequality, for any $\vec x\in \mathbb{R}^{N_h}$ and $p\ge 1$ we always have
$$
\left(\frac{\sum_{i=1}^{N_h} |x_i|^p}{N_h}\right)^{\frac{1}{p}}\ge \frac{\sum_{i=1}^{N_h} |x_i|}{N_h}. 
$$
Combining this with \eqref{Distance1}, we have
\beq
\begin{aligned}
\big| \hat X_{h,i}(s)-X_{h,i}(s)\big|_{l_h^1}&\le \big(h^d N_h \big)^{1-p^{-1}}\big| \hat X_{h,i}(s)-X_{h,i}(s)\big|_{l_h^p}\\  
&\le U_D \big| \hat X_{h,i}(s)-X_{h,i}(s)\big|_{l_h^p}. 
\end{aligned}
\eeq
Plugging this inequality to \eqref{L23} and according to the definition of $E_h$, we have under event $E_h$:
\beq
\label{L24}
\begin{aligned}
\left\|\rho_{h}(x,s)-\hat \rho_{h}(x,s)\right\| &\le 2c_d\|\rho_0\|_\infty \max_{k=1,\cdots,d}\left\|\frac{\partial\varphi}{\partial x_k}\right\|_\infty \ep_h^{-d/2-1}U_D C' h |\ln h|\\
&\le 2c_d\|\rho_0\|_\infty \max_{k=1,\cdots,d}\left\|\frac{\partial\varphi}{\partial x_k}\right\|_\infty \ep_h^{-d/2-1}U_D C' h^{2/3}
\end{aligned}
\eeq
when $h$ is sufficiently small. Noting that $\ep_h=h^{1/6d}$, we have 
$$
\ep_h^{-d/2-1} h^{2/3}=h^{7/12-1/6d}< h^{1/12d}. 
$$
Combining this observation with the fact that \eqref{L24} holds true for all $s\in [0,t]$, we have 
\beq
\label{L25}
\sup_{s\in [0,t]}\left\|\rho_{h}(x,s)-\hat \rho_{h}(x,s)\right\| \le 2c_d\|\rho_0\|_\infty \max_{k=1,\cdots,d}\left\|\frac{\partial\varphi}{\partial x_k}\right\|_\infty U_D C' h^{1/6d}. 
\eeq
Then similarly for any $s\in [0,t]$ consider 
$$
\left\|\nabla\rho_{h}(x,s)-\nabla\hat \rho_{h}(x,s)\right\|\le \sum_{k=1}^d \left\|\frac{\partial \rho_{h}(x,s)}{\partial x_k}-\frac{\partial \hat \rho_{h}(x,s)}{\partial x_k}\right\|.
$$
Then for each $i$ we can similarly we have 
\begin{align*}
\left\|\frac{\partial \rho_{h}(x,s)}{\partial x_k}-\frac{\partial \hat \rho_{h}(x,s)}{\partial x_k}\right\|&\le \sum_{i=1}^{N_h} h^d \rho_0(\theta_{i,h}) \left\| \frac{\partial \varphi_{\ep_h}\Big(x-X_{h,i}(s)\Big)}{\partial x_k}-\frac{\partial \varphi_{\ep_h}\left(x-\hat X_{h,i}(s)\right) }{\partial x_k}\right\|\\
&=\sum_{i=1}^{N_h} h^d \ep_{h}^{-d-1} \rho_0(\theta_{i,h}) \left\| \frac{\partial \varphi}{\partial x_k}\left(\frac{x-X_{h,i}(s)}{\ep_h}\right)-\frac{\partial \varphi}{\partial x_k}\left(\frac{x-\hat X_{h,i}(s)}{\ep_h}\right) \right\|.
\end{align*}
Again according to mid-value theorem, we have 
$$
\left| \frac{\partial \varphi}{\partial x_k}\left(\frac{x-X_{h,i}(s)}{\ep_h}\right)-\frac{\partial \varphi}{\partial x_k}\left(\frac{x-\hat X_{h,i}(s)}{\ep_h}\right) \right|\le c_d\left| \frac{X_{h,i}(s)-\hat X_{h,i}(s)}{\ep_h}\right|\max_{j,k=1,\cdots,d}\left\|\frac{\partial^2\varphi}{\partial x_j \partial x_k}\right\|_\infty
$$
and 
$$
\left| \frac{\partial \varphi}{\partial x_k}\left(\frac{x-X_{h,i}(s)}{\ep_h}\right)-\frac{\partial \varphi}{\partial x_k}\left(\frac{x-\hat X_{h,i}(s)}{\ep_h}\right) \right|\equiv 0
$$
for all $x\notin U_i=\{y: \big|y-X_{h,i}\big|\le \ep_h/2\}\cup \{y: \big|y-\hat X_{h,i}\big|\le \ep_h/2\}$. Thus 
$$
\begin{aligned}
\left\| \frac{\partial \varphi}{\partial x_k}\left(\frac{x-X_{h,i}(s)}{\ep_h}\right)-\frac{\partial \varphi}{\partial x_k}\left(\frac{x-\hat X_{h,i}(s)}{\ep_h}\right) \right\|\le 2c_d\left| \frac{X_{h,i}(s)-\hat X_{h,i}(s)}{\ep_h}\right| \max_{j,k=1,\cdots,d}\left\|\frac{\partial^2\varphi}{\partial x_j \partial x_k}\right\|_\infty \ep_h^{d/2}
\end{aligned}
$$
which implies that 
\beq
\label{H11}
\left\|\frac{\partial \rho_{h}(x,s)}{\partial x_k}-\frac{\partial \hat \rho_{h}(x,s)}{\partial x_k}\right\|\le 2c_d\|\rho_0\|_\infty \max_{j,k=1,\cdots,d}\left\|\frac{\partial^2\varphi}{\partial x_j\partial x_k}\right\|_\infty \ep_h^{-d/2-2}U_D C' h^{2/3}
\eeq
when $h$ is sufficiently small. Again noting that $\ep_h=h^{1/6d}$, we have 
$$
\ep_h^{-d/2-2} h^{2/3}=h^{7/12-1/3d}< h^{1/12d}.
$$
And note that \eqref{H11} holds for all $k=1,2,\cdots, d$. Thus for any $s\in [0,t]$, 
\beq
\label{H12}
\left\|\nabla\rho_{h}(x,s)-\nabla\hat \rho_{h}(x,s)\right\|\le 2dc_d\|\rho_0\|_\infty \max_{j,k=1,\cdots,d}\left\|\frac{\partial^2\varphi}{\partial x_j\partial x_k}\right\|_\infty U_D C' h^{1/6d}
\eeq
and 
\beq
\label{H13}
\int_{s=0}^t\left\|\nabla\rho_{h}(x,s)-\nabla\hat \rho_{h}(x,s)\right\| ds\le 2dtc_d\|\rho_0\|_\infty \max_{j,k=1,\cdots,d}\left\|\frac{\partial^2\varphi}{\partial x_j\partial x_k}\right\|_\infty U_D C' h^{1/6d}. 
\eeq
Let 
\beq
\label{C0}
c_0(t)=2dtc_d\|\rho_0\|_\infty \max_{j,k=1,\cdots,d}\left\|\frac{\partial^2\varphi}{\partial x_j\partial x_k}\right\|_\infty U_D C'+ 2c_d\|\rho_0\|_\infty \max_{k=1,\cdots,d}\left\|\frac{\partial\varphi}{\partial x_k}\right\|_\infty U_D C'+1.
\eeq
It is easy to see that 
$$
h^{\Lambda C' |\ln h|}<h^{1/12d}<h^{1/12d}c_0(t)
$$
when $h$ is sufficiently small. Thus the proof is complete. 
\end{proof}

\section{Decomposition of Errors}
Since we have estimated the distance between the empirical measures constructed from the interacting particle system and the self-consistent system. The remainder of the paper will mostly devote to the proof of Theorem 2. First, as described in the outline of the proof, we use Ito's formula to separate this distance into a term of the free energy estimation of PDE,  a term of initial error, a truncation error and a martingale error. To be precise, we have a proposition as follows:

\begin{prop} 
\label{Decompose}
For the difference between the PDE density $\rho$ and the empirical measure $\hat \rho_{h}$ constructed from the self-consistent system, we have for any $s\in [0,t]$
\beq
\label{Prop1}
\begin{aligned}
\| (\rho -\hat\rho_{h})(\cdot,s)\|^2 
 & =\| (\rho-\hat\rho_{h})(\cdot,0) \|^2
-\int_0^s\|\nabla (\rho-\hat\rho_{h})(\cdot,q) \|^2\, dq\\
&\quad
-  \int_0^s \int_{{\mathbb R}^d} 
 \nabla\cdot  F(x,q)   (\rho - \hat\rho_{h})^2(x,q) \, dx dq\\
&\quad
+Tr (s) +\bar M_s
\end{aligned}
\eeq
where $\| (\rho-\hat\rho_{h})(\cdot,0) \|^2$ is the initial error and the term of 
$$
-\int_0^s\|\nabla (\rho-\hat\rho_{h})(\cdot,q) \|^2\, dq-  \int_0^s \int_{{\mathbb R}^d} \nabla\cdot F(x,q)  \left( (\rho - \hat\rho_{h})(x,q)\right)^2 \, dx dq
$$
gives the free energy estimation. $\bar M_s=M_s+\tilde M_s$ is the martingale error from the Ito's formula, where $M_s$ is defined by $M_s=\sum_{i=1}^{N_h} \rho_0(\theta_{i,h}h) M_s^i$ with 
\beq
\label{Prop11}
M_s^i= 2h^d \int_0^s \int_{{\mathbb R}^d}\rho(x,q) \nabla \varphi_{\ep_h}\left(x-\hat X_{h,i}(q)\right)dx\cdot d B_i(q),
\eeq
and $\tilde M_s=\sum_{n=1}^{N_h}\rho_0(\theta_{i,h}h) \tilde M_s^i$ with $\tilde M_s^i$ equals to 
\beq
\label{Prop13}
\begin{aligned}
&2h^{2d} \int_0^s \int_{{\mathbb R}^d}\varphi_{\ep_h}(x)
\left( \sum_{j=1}^{i-1}\rho_0(\theta_{j,h}h) \nabla\varphi_{\ep_h}\left(x+\hat X_{h,i}(q)-\hat X_{h,j}(q)\right)\right.\\
& \hspace{2 in} \left.- \sum_{j=i+1}^{N_h}\rho_0(\theta_{j,h}h) \nabla\varphi_{\ep_h}\left(x+\hat X_{h,j}(q)-\hat X_{h,i}(q)\right)\right)\, dx\cdot dB_i(q).
\end{aligned}
\eeq
And $Tr(s)$ is the term of truncation error which is defined as 
\begin{align*}
\label{Prop14}
Tr(s)
=&2 \int_0^s \int_{{\mathbb R}^d} 
\left[\sum_{i=1}^{N_h} h^d\rho_0(\theta_{i,h}h) \varphi_{\ep_h}\left( x-\hat X_{h,i}(q)\right)
\left( F(x,q) - F\left(\hat X_{h,i}(q),q\right)  \right) \right]
\cdot \nabla  (\rho -\hat \rho_{h})(x,q) \, dx dq\\
&+h^{2d}s \|\nabla \varphi_{\ep_h}\|^2 \sum_{i=1}^{N_h} \rho_0(\theta_{i,h}h)^2. 
\end{align*}
\end{prop}

\begin{proof}To prove the proposition, first note that for any $h$,
$$
\| (\rho-\rho_{h})(\cdot,s) \|^2
=\| \rho(\cdot,s)\|^2 
- 2 \int_{{\mathbb R}^d} \rho(x,s) \hat\rho_{h}(x,s)\, dx+\| \hat\rho_{h}(\cdot,s)\|^2.
$$
First for the deterministic part of $\| \rho(\cdot,t)\|^2$, we have 
\beq
\label{DT}
\begin{aligned}
\int_{{\mathbb R}^d}  \rho(x,s)^2 \, dx=&\int_{{\mathbb R}^d}  \rho(x,0)^2 \, dx
+\int_0^s \int_{{\mathbb R}^d} \rho(x,q)\left( \Delta \rho(x,q)-2\nabla \cdot (\rho  F) (x,q) \right)\, dxdq \\
=&\int_{{\mathbb R}^d}  \rho(x,0)^2 \, dx-\int_0^s\|\nabla \rho\|^2dq+2\int_0^s \int_{\mathbb{R}^d} \rho(x,q)F(x,q)\cdot \nabla \rho(x,q) dx dq
\end{aligned}
\eeq
Then for the second part which equals to 
$$
-2h^d\int_{{\mathbb R}^d} \rho(x,s)\sum_{i=1}^{N_h}\rho_0(\theta_{i,h}h)\varphi_{\ep_h}\left(x-\hat X_{h,i}(s)\right)dx,
$$
note that for each $i$, by Ito's formula, we have 
\beq
\label{ITOD}
\begin{aligned}
\rho(x,s)\varphi_{\ep_h}\left(x-\hat X_{h,i}(s)\right)&=\rho(x,0)\varphi_{\ep_h}\left(x-\hat X_{h,i}(0)\right)+\int_0^s \frac{\partial \rho(x,q)}{\partial t} \varphi_{\ep_h}\left(x-\hat X_{h,i}(q)\right)\,dq\\
&-\int_0^s \rho(x,q) \nabla \varphi_{\ep_h}\left(x-\hat X_{h,i}(q)\right)\cdot  F\left(\hat X_{h,i}(q),q \right)\,dq\\
&-\int_0^s \rho(x,q) \nabla \varphi_{\ep_h}\left(x-\hat X_{h,i}(q)\right)\cdot d B_i(q)\\
&+\frac{1}{2}\int_0^s \rho(x,q) \Delta \varphi_{\ep_h}\left(x-\hat X_{h,i}(q)\right)\, dq.
\end{aligned}
\eeq
Note that for the second term in the right hand side of the sum above, according to the definition of our PDE, 
$$
\int_0^s \frac{\partial \rho(x,q)}{\partial t} \varphi_{\ep_h}\left(x-\hat X_{h,i}(q)\right)\, dq
=\int_0^s\left( \frac{1}{2}\Delta \rho(x,q)-\nabla \cdot \big(\rho(x,q) F(x,q)\big)\right)\varphi_{\ep_h}\left(x-\hat X_{h,i}(q)\right)\, dq.
$$
Then integrate it over $x\in \mathbb{R}^d$, we have 
\beq
\label{Mix1}
\begin{aligned}
\int_0^s&\int_{{\mathbb R}^d}  \frac{\partial \rho(x,q)}{\partial t} \varphi_{\ep_h}\left(x-\hat X_{h,i}(q)\right)dx dq\\
=&-\frac{1}{2}\int_0^s \int_{{\mathbb R}^d} \nabla \rho(x,q)\cdot \nabla \varphi_{\ep_h}\left(x-\hat X_{h,i}(q)\right)dx dq\\
&+ \int_0^s \int_{{\mathbb R}^d} \rho(x,q) F(x,q) \cdot \nabla \varphi_{\ep_h}\left( x-\hat X_{h,i}(q)\right) \, dx dq.
%\\
%&-\int_0^t \int_{{\mathbb R}^d}  \left[ \nabla \rho(x,s)\cdot \vec F(x,s)\right] \varphi_{\ep_h}\left(x-X_n(s)\right)dx ds.
\end{aligned}
\eeq
Then integrating the third term in (\ref{ITOD}) over $x\in \mathbb{R}^d$ we have by divergence theorem that
\beq
\label{Mix2}
\begin{aligned}
-\int_0^s&\int_{{\mathbb R}^d} \rho(x,q) \nabla \varphi_{\ep_h}\left(x-\hat X_{h,i}(q)\right)\cdot F\left(\hat X_{h,i}(q),q \right)dxdq\\
&=\int_0^s \int_{{\mathbb R}^d}  \varphi_{\ep_h}\left(x-\hat X_{h,i}(q)\right) F(\hat X_{h,i}(q),q) \cdot \nabla \rho(x,q)\, dx dq.
\end{aligned}
\eeq
Combining (\ref{ITOD}), (\ref{Mix1}) and (\ref{Mix2}) we have
\beq
\begin{aligned}
\int_{{\mathbb R}^d} &\rho(x,s)\varphi_{\ep_h}\left(x-\hat X_{h,i}(s)\right)dx\\
&=\int_{{\mathbb R}^d}  \rho(x,0) \varphi_{\ep_h}\left(x-\hat X_{h,i}(0)\right)dx\\
&-\int_0^s \int_{{\mathbb R}^d} \nabla \rho(x,q)\cdot \nabla \varphi_{\ep_h}\left(x-\hat X_{h,i}(q)\right)dx dq\\
%&-\int_0^t \int_{{\mathbb R}^d} \rho(x,s)\varphi_{\ep_h}\left( x-X_n(s)\right)\nabla\cdot \vec F(x,s) dx ds\\
&+\int_0^s \int_{{\mathbb R}^d}  \rho(x,q)  F(x,q)  \cdot \nabla \varphi_{\ep_h}\left(x-\hat X_{h,i}(q)\right)dx dq \\
&+\int_0^s \int_{{\mathbb R}^d}   \varphi_{\ep_h}\left(x-\hat X_{h,i}(q)\right)  F(\hat X_{h,i}(q),q)  \cdot \nabla \rho(x,q) dx dq
- \frac{h^{-d}}{2}M^i_s
\end{aligned}
\eeq
where $M^i_s$ is a martingale given in \eqref{Prop11}, i.e.,  
$$
M^i_s=2h^{d}\int_0^s \int_{{\mathbb R}^d}\rho(x,q) \nabla \varphi_{\ep_h}\left(x-\hat X_i(q)\right)dx\cdot d B_i(q). 
$$
Summing up and taking the weighted average over $i=1,2,\cdots,N_h$, we have
\beq
\label{Mix}
\begin{aligned}
-2\int_{{\mathbb R}^d}& \hat\rho_{h}(x,s) \rho(x,s)\, dx \\
&=-2 \int_{{\mathbb R}^d} \hat\rho_{h}(x,0) \rho(x,0)\,dx
  +2\int_0^s \int_{{\mathbb R}^d} \nabla \rho(x,q) \cdot \nabla \hat\rho_{h}(x,q)\,dx\, dq\\
%&+2\int_0^t \int_{{\mathbb R}^d} \rho(x,s)\rho_{\ep,N}\left( x-X_n(s)\right)\nabla\cdot \vec F(x,s) dx ds\\
&-2\int_0^s \int_{{\mathbb R}^d}   \rho(x,q) F(x,q) \cdot  \nabla\hat\rho_{h}(x,q)\, dx\, dq\\
&-2\int_0^s \int_{{\mathbb R}^d}  h^d\sum_{i=1}^{N_h}\rho_0(\theta_{i,h}h)
\varphi_{\ep_h}\left(x-\hat X_{h,i}(q)\right)  F\left(\hat X_{h,i}(q),q\right)  \cdot  \nabla  \rho(x,q)\, dx\, dq\\
&+M_s
\end{aligned}
\eeq
Lastly, we look at the part of $\left\|\hat\rho_{h}(\cdot,s)\right\|^2$ which equals to 
\beq
\begin{aligned}
h^{2d}\sum_{i,j=1,2,\cdots,N_h}\rho_0(\theta_{i,h}h)\rho_0(\theta_{j,h}h)\int_{{\mathbb R}^d} \varphi_{\ep_h}\left( x-\hat X_{h,i}(s)\right) \varphi_{\ep_h}\left( x-\hat X_{h,j}(s)\right)\, dx. 
\end{aligned}
\eeq
For each $i,j\in \{1,2,\cdots,N_h\}$, if $i=j$, we have directly from change of variables that 
$$
\int_{{\mathbb R}^d} \varphi_{\ep_h}\left( x-\hat X_{h,i}(s)\right) \varphi_{\ep_h}\left( x-\hat X_{h,i}(s)\right)\, dx= \|\varphi_{\ep_h}\|^2. 
$$
And if $i\not=j$, say without loss of generality $i<j$ again by change of variables we have 
$$
\int_{{\mathbb R}^d} \varphi_{\ep_h}\left( x-\hat X_{h,i}(s)\right) \varphi_{\ep_h}\left( x-\hat X_{h,j}(s)\right)\, dx
=\int_{{\mathbb R}^d} \varphi_{\ep_h}\left( x\right) \varphi_{\ep_h}\left( x+\hat X_{h,j}(s)-\hat X_{h,i}(s)\right)\, dx.
$$
Then we can again apply the Ito's formula on $\varphi_{\ep_h}\left( x+\hat X_{h,j}(s)-\hat X_{h,i}(s)\right)$ and have it equals to:
\beq
\label{E1}
\begin{aligned}
\varphi_{\ep_h}&\left( x+\hat X_{h,j}(0)-\hat X_{h,i}(0)\right)+\int_0^s \Delta \varphi_{\ep_h}\left( x+\hat X_{h,j}(q)-\hat X_{h,i}(q)\right)\, dq\\
&+\int_0^s \nabla\varphi_{\ep_h}\left( x+\hat X_{h,j}(q)-\hat X_{h,i}(q)\right)\cdot \left(F\left(\hat X_{h,j}(q),q\right)- F\left(\hat X_{h,i}(q),q\right)\right)\, dq\\
&+\int_0^s \nabla \varphi_{\ep_h}\left( x+\hat X_{h,j}(q)-\hat X_{h,i}(q)\right) \cdot \big(dB_j(q)-dB_i(q) \big). 
\end{aligned}
\eeq
Integrating the first and second terms over $x$, we have
$$
\int_{{\mathbb R}^d} \varphi_{\ep_h}(x)\varphi_{\ep_h}\left( x+\hat X_{h,j}(0)-\hat X_{h,i}(0)\right)dx=\int_{{\mathbb R}^d} \varphi_{\ep_h}\left(x-\hat X_{h,i}(0)\right)\varphi_{\ep_h}\left( x-\hat X_{h,j}(0)\right)dx
$$
and
\begin{align*}
\int_0^s\int_{{\mathbb R}^d}&\varphi_{\ep_h}\left(x\right) \Delta \varphi_{\ep_h}\left( x+\hat X_{h,j}(q)-\hat X_{h,i}(q)\right)dxdq\\
&=-\int_0^s \int_{{\mathbb R}^d}\nabla \varphi_{\ep_h}\left(x-\hat X_{h,i}(q)\right) \cdot\nabla \varphi_{\ep_h}\left( x-\hat X_{h,j}(q)\right)dxdq.
\end{align*}
Moreover, for the third term we have that
\begin{align*}
\int_0^s\int_{{\mathbb R}^d}&\varphi_{\ep_h}(x) \nabla\varphi_{\ep_h}\left( x+\hat X_{h,j}(q)-\hat X_{h,i}(q)\right)\cdot F\left(\hat X_{h,j}(q),q\right)dxdq\\
&=\int_0^s\int_{{\mathbb R}^d}
\varphi_{\ep_h}\left( x-\hat X_{h,j}(q)\right)F\left(\hat X_{h,j}(q),q\right)
\cdot \nabla\varphi_{\ep_h}\left( x-\hat X_{h,i}(q)\right)\, dxdq
\end{align*}
and that 
\begin{align*}
\int_0^s\int_{{\mathbb R}^d}&\varphi_{\ep_h}(x) \nabla\varphi_{\ep_h}\left( x+\hat X_{h,j}(q)-\hat X_{h,i}(q)\right)\cdot F\left(X_{h,i}(q),q\right)dxdq\\
&=\int_0^s\int_{{\mathbb R}^d}
\varphi_{\ep_h}\left( x-\hat X_{h,j}(q)\right)F\left(\hat X_{h,i}(q),q\right)
\cdot \nabla\varphi_{\ep_h}\left( x-\hat X_{h,i}(q)\right)\, dxdq\\
&=-\int_0^s\int_{{\mathbb R}^d}
\varphi_{\ep_h}\left( x-\hat X_{h,i}(q)\right)F\left(\hat X_{h,i}(q),q\right)\cdot 
\nabla\varphi_{\ep_h}\left( x-\hat X_{h,j}(q)\right)\, dxdq
\end{align*}
by divergence theorem. We also note that
\beq
\label{E2}
\int_{{\mathbb R}^d}\varphi_{\ep_h}\left(x-\hat X_{h,i}(s)\right) \nabla\varphi_{\ep}\left( x-\hat X_{h,i}(s)\right)\cdot \vec F\left(\hat X_{h,i}(s),s\right)dx\equiv 0. 
\eeq
So after we sum up over all the $i,j$ and have the weighted average, for combinations of the initial values for $i\not=j$ and the constant values for $i=j$ we have 
\beq
\label{E3}
\begin{aligned}
2\sum_{i<j\in \{1,2,\cdots,N_h\}}&h^{2d}\rho_0(\theta_{i,h}h)\rho_0(\theta_{j,h}h)\int_{{\mathbb R}^d} \varphi_{\ep_h}\left(x-\hat X_{h,i}(0)\right)\varphi_{\ep_h}\left( x-\hat X_{h,j}(0)\right)dx\\
&+\sum_{i=1}^{N_h}h^{2d}\rho_0(\theta_{i,h}h)^2\|\varphi_{\ep_h}\|^2=\int_{\mathbb{R}^d} \hat\rho_{h}(x,0)^2 dx=\left\|\hat\rho_{h}(\cdot,0)\right\|^2
\end{aligned}
\eeq
Then summing up the integration over $x$ of the second term in \eqref{E1}, we have
\beq
\label{E4}
\begin{aligned}
-2\sum_{i<j\in \{1,2,\cdots,N_h\}}&h^{2d}\rho_0(\theta_{i,h}h)\rho_0(\theta_{j,h}h)\int_0^s \int_{{\mathbb R}^d}\nabla \varphi_{\ep_h}\left(x-\hat X_{h,i}(q)\right) \cdot\nabla \varphi_{\ep_h}\left( x-\hat X_{h,j}(q)\right)dxdq\\
&=-\int_0^s \left\| \nabla \hat \rho_h(\cdot,q)\right\|^2 dq+h^{2d}s \|\nabla \varphi_{\ep_h}\|^2 \sum_{i=1}^{N_h} \rho_0(\theta_{i,h}h)^2. 
\end{aligned}
\eeq
Then summing up the integration over $x$ of the third term in \eqref{E1} and note that we can add the zero terms in \eqref{E2} for each $i$ in the weighted average, we have 
\beq
\label{E5}
\begin{aligned}
2\sum_{i<j\in \{1,2,\cdots,N_h\}}h^{2d}\rho_0(\theta_{i,h}h)\rho_0(\theta_{j,h}h)\int_0^s \int_{\mathbb{R}^d}&\varphi_{\ep_h}(x)\nabla\varphi_{\ep_h}\left( x+\hat X_{h,j}(q)-\hat X_{h,i}(q)\right)\\
&\cdot \left(F\left(\hat X_{h,j}(q),q\right)- F\left(\hat X_{h,i}(q),q\right)\right)\, dx dq\\
=2 \int_0^s \int_{{\mathbb R}^d} \sum_{i=1}^{N_h} h^d\rho_0(\theta_{i,h}h) \varphi_{\ep_h}&\left( x-\hat X_{h,i}(q)\right)F\left(\hat X_{h,i}(q),q\right)\cdot \nabla \hat \rho_{h}(x,q) \, dx dq
\end{aligned}
\eeq
And lastly, if we sum up the last term in \eqref{E1} for all the $i,j$'s, we just get $\tilde M_s=\sum_{n=1}^{N_h}\tilde\rho_0(\theta_{i,h}h) \tilde M_s^i$ with $\tilde M_s^i$ defined in \eqref{Prop13}. Thus combining \eqref{E3}-\eqref{E5}, 
\beq
\label{E6}
\begin{aligned}
\left\|\hat\rho_{h}(\cdot,s)\right\|^2=&\left\|\hat\rho_{h}(\cdot,0)\right\|^2-\int_0^s \left\| \nabla \hat \rho_h(\cdot,q)\right\|^2 dq\\
&+2 \int_0^s\int_{{\mathbb R}^d} \sum_{i=1}^{N_h} h^d\rho_0(\theta_{i,h}h) \varphi_{\ep_h}\left( x-\hat X_{h,i}(q)\right)F\left(\hat X_{h,i}(q),q\right)\cdot \nabla \hat \rho_{h}(x,q) \, dx dq\\
&+\tilde M_s+h^{2d}s \|\nabla \varphi_{\ep_h}\|^2 \sum_{i=1}^{N_h} \rho_0(\theta_{i,h}h)^2. 
\end{aligned}
\eeq
At this point, we can combine \eqref{DT}, \eqref{Mix} and \eqref{E6} and have 
\beq
\label{Prop15}
\begin{aligned}
\| (\rho -\hat\rho_{h})(\cdot,s)\|^2 
 & =\| (\rho-\hat\rho_{h})(\cdot,0) \|^2
-\int_0^s\|\nabla (\rho-\hat\rho_{h})(\cdot,q) \|^2\, dq+\bar M_s\\
&+h^{2d}s \|\nabla \varphi_{\ep_h}\|^2 \sum_{i=1}^{N_h} \rho_0(\theta_{i,h}h)^2\\
&+2 \int_0^s\int_{{\mathbb R}^d} \sum_{i=1}^{N_h} h^d\rho_0(\theta_{i,h}h) \varphi_{\ep_h}\left( x-\hat X_{h,i}(q)\right)F\left(\hat X_{h,i}(q),q\right)\cdot \nabla \hat \rho_{h}(x,q) \, dx dq\\
&+2\int_0^s \int_{\mathbb{R}^d} \rho(x,q)F(x,q)\cdot \nabla \rho(x,q) dx dq\\
&-2\int_0^s \int_{{\mathbb R}^d}   \rho(x,q) F(x,q) \cdot  \nabla\hat\rho_{h}(x,q)\, dx\, dq\\
&-2\int_0^s \int_{{\mathbb R}^d}  \sum_{i=1}^{N_h}h^d\rho_0(\theta_{i,h}h)
\varphi_{\ep_h}\left(x-\hat X_{h,i}(q)\right)  F\left(\hat X_{h,i}(q),q\right)  \cdot  \nabla  \rho(x,q)\, dx\, dq.
\end{aligned}
\eeq
Then plus and minus the term 
\beq
\label{add}
2\int_0^s\int_{\mathbb{R}^d}\sum_{i=1}^{N_h} h^{d}\rho_0(\theta_{i,h}h) \varphi_{\ep_h}\left( x-\hat X_{h,i}(q)\right) F\left(x,q\right)\cdot \nabla (\rho - \hat \rho_{h})(x,q)\, dx dq
\eeq
we have 
\beq
\label{Prop16}
\begin{aligned}
\| (\rho -\hat\rho_{h})(\cdot,s)\|^2 
  =&\| (\rho-\hat\rho_{h})(\cdot,0) \|^2
-\int_0^s\|\nabla (\rho-\hat\rho_{h})(\cdot,q) \|^2\, dq+Tr (s) +\bar M_s\\
&+2\int_0^s\int_{\mathbb{R}^d} (\rho - \hat \rho_{h})(x,q) F\left(x,q\right)\cdot \nabla (\rho - \hat \rho_{h})(x,q)\, dx dq. 
\end{aligned}
\eeq
By Green's theorem, 
\begin{align*}
2\int_0^s\int_{\mathbb{R}^d}& (\rho - \hat \rho_{h})(x,q) F\left(x,q\right)\cdot \nabla (\rho - \hat \rho_{h})(x,q)\, dx dq\\
&=-\int_0^s\int_{\mathbb{R}^d} \nabla\cdot F\left(x,q\right)(\rho - \hat \rho_{h})^2(x,q)\, dxdq.
\end{align*}
We have verified \eqref{Prop1} and the proof of Proposition 1 is complete. 
\end{proof}

\section{Estimation on the Separation}
With Proposition \ref{Decompose} decomposing the distance between $\rho$ and $\hat \rho_{h}$ as the sum of several different error terms with different physical and mathematical meanings, we will estimate those error terms one by one. But first, we prove an estimation of separations which shows that, with our initial data, the independent solutions of the self-consistent SDE with high probability cannot be too close to each other. To be specific, we use
\beq
E_j(t)=\frac{1}{N_h}\sum_{i\le N_h: i\not=j} \int_0^t P( |\hat X_{h,i}(s)-\hat X_{h,j}(s)| \le 2\ep_h)\, ds
\label{R*2}
\eeq
to measure the separation of the self-consistent system. Intuitively, we can see $E_j(t)$ as the sum of the average length of time for each particle that is within a distance of $2\ep_h$ from particle $j$. And we have the following proposition showing that $E_j(t)$ is small, which implies that, with high probability, the path of different particles in the self-consistent system cannot be too close to each other. The reason we want to first prove the proposition can be seen later in \eqref{R*1}. 
\begin{prop}
\label{Lemma 4.1}
There exist some constants $C_1(t)$ and $C_2(t)$ depends only on $t$, $d$ and $F_0$ such that 
\beq
\label{lemma 4.1}
E_j(t)\le C_1(t)\ep_h^{d-1}+C_2(t)\frac{1}{N_h\ep_h}
\eeq 
for all $j=1,2,\cdots, N_h$, when $h$ is sufficiently small. 
\end{prop}
\begin{proof}
For any $h$ and $j\le N_h$, Fix $i\not=j$, $i\le N_h$, and let $\{\Omega,\mathcal{F}^{i,j}_t,P\}$ be our probability measure space where $,\mathcal{F}^{i,j}_t$ is the natural filtration generated by $B^*_{i,j}(t)=[B_i(t), B_j(t)]$, which is a $2d$-dimensional Brownian motion. Let $\theta_{i,j}(s)=- \big( F(\hat X_{h,i}(s),s), F(\hat X_{h,j}(s),s) \big)$ be the integrand and consider the adapted measurable process 
\beq
\Gamma_s=\int_0^s \theta_{i,j}(h) \cdot d B^*_{i,j}(h).
\eeq 
Note that for any $s\ge 0$, 
\beq
\label{new1}
|\theta_{i,j}(s)|^2\le 2d \,\|F_0\|_\infty^2. 
\eeq
Thus the Novikov condition (see page 198 of \cite{KS} for details) is satisfied, i.e., 
$$
E\left[\exp\left(\frac{1}{2} \int_0^s |\theta_{i,j}(q)|^2 \,dq \right)\right]\le \exp(sd\, \|F_0\|_\infty^2 )<\infty,
$$
by Girsanov Theorem (see Theorem 3.5.1 of \cite{KS}) we can define a probability measure $Q$ in our probability space with Radon-Nikodym derivative
\beq
\left. \frac{d Q_{i,j}}{d P}\right |\mathcal{F}_s=\mathcal{E}_s
=\exp\left[ \Gamma_s-\frac{1}{2} \int_0^s|\theta_{i,j}(q)|^2 \,dq \right].
\eeq
Then we have 
$$
\left[
\begin{array}{c}
\hat X_{h,i}(s)-\hat X_{h,i}(0)\\
\hat X_{h,j}(s)-\hat X_{h,j}(0)
\end{array}
\right]=
\left[
\begin{array}{c}
B_i(s)+\int_0^s F(\hat X_{h,i}(q),q)dq\\
B_j(s)+\int_0^s F(\hat X_{h,j}(q),q)dq
\end{array}
\right]=
B_{i,j}^*(s)-\langle \Gamma, B_{i,j}^*\rangle_s
$$
is a standard $2d$-dimensional Brownian motion under probability measure $Q_{i,j}$, where $\langle \Gamma, B_{i,j}^*\rangle_s$ is again the quadratic covariance between $\Gamma_s$ and $B_{i,j}^*(s)$. Thus by Radon-Nikodym Theorem we have 
\beq
\label{Gir1}
\int_{ |\hat X_{h,i}(s)-\hat X_{h,j}(s)| \le 2\ep_h} \mathcal{E}_s dP= P\left( |B_i(s)-B_j(s)+\gamma_{i,j}|
\le 2\ep_h\right)
\eeq
where $\gamma_{i,j}=h[\theta(j,h)-\theta(i,h)]$. Moreover, 
$$
P\left( |\hat X_{h,i}(s)-\hat X_{h,j}(s)| \le 2\ep_h\right)
\le P(\mathcal{E}_s <\ep_h^{1/2})+P\left( |\hat X_{h,i}(s)-\hat X_{h,j}(s)| \le 2\ep_h\, \cap \, \mathcal{E}_s \ge \ep_h^{1/2}\right)
$$
and for the first part we have,
\beq
\label{AS11}
P(\mathcal{E}_s <\ep_h^{1/2})\le P\left( \exp\left[\int_0^s (-\theta_{i,j}(q)) \cdot d B^*_{i,j}(q)\right]>\ep_h^{-1/2}\exp(-sd\, \|F_0\|_\infty^2 )\right).
\eeq
To control the right hand side of the inequality above, we consider the $L^{4d}$ norm:
\beq
E\left[\left(\exp\left[\int_0^s (-\theta_{i,j}(q)) \cdot d B_{i,j}^*(q)\right]\right)^{4d}\right]
=E\left(\exp\left[\int_0^s (-4d\, \theta_{i,j}(q)) \cdot d B_{i,j}^*(q)\right]\right)
\eeq
and note that again by Girsanov Theorem,
$$
\mathcal{E}'_s= \exp\left[\int_0^s (-4d\theta_{i,j}(q)) \cdot d B_{i,j}^*(q)\right]\exp\left(-8d^2\int_0^s |\theta_{i,j}(q)|^2 \, dq \right)
$$
is again a Radon-Nikodym derivative. Thus we have 
$$
E(\mathcal{E}'_s)=1
$$
which combining with (\ref{new1}), implies 
\beq
\label{AS12}
E\left(\exp\left[\int_0^s (-4d\theta_{i,j}(q)) \cdot d B_{i,j}^*(q)\right]\right)\le \exp(16d^3s\,  \|F_0\|_\infty^2 )<\infty.
\eeq
Combining (\ref{AS11}), (\ref{AS12}) and Chebyshev's Inequality gives us 
\beq
P(\mathcal{E}_s <\ep_h^{1/2})\le \ep_h^{2d}\exp\left( (4d^2+16d^3)s\, \|F_0\|_\infty^2 \right).
\eeq 
Then for the second part, according to (\ref{Gir1}) we have
$$
\int_{ |\hat X_{h,i}(s)-\hat X_{h,j}(s)| \le 2\ep_h \, \cap \, \mathcal{E}_s \ge \ep_h^{1/2}} \mathcal{E}_s dP\le P\left( |B_i(s)-B_j(s)+\gamma_{i,j}|
\le 2\ep_h\right). 
$$
and thus 
\beq
\begin{aligned}
P&\left( |\hat X_{h,i}(s)-\hat X_{h,j}(s)| \le 2\ep_h \, \cap \, \mathcal{E}_s \ge \ep_h^{1/2}\right)\\
&\le \ep_h^{-1/2}P\left(  |B_i(s)-B_j(s)+\gamma_{i,j}|  \le 2\ep_h\right).
\end{aligned}
\eeq
Combining the two inequalities above, we have 
\beq
\label{AS13}
\begin{aligned}
P\left( |\hat X_{h,i}(s)-\hat X_{h,j}(s)|\le 2\ep_h\right)\le & 
\, \ep^{2d}\exp\left( (4d^2+16d^3)s\, \|F_0\|_\infty^2 \right)\\
&+\ep_h^{-1/2}P\left(  |B_i(s)-B_j(s)+\gamma_{i,j}|\le 2\ep_h\right)
\end{aligned}
\eeq
for any $s\ge 0$. Integrating (\ref{AS13}) on $[0,t]$ and averaging over all $i\not=j$, $i\le N_h$, we have 
\beq
\label{AS14}
\begin{aligned}
E_j(t)\le & \frac{\ep_h^{2d}}{(4d^2+16d^3)\, \|F_0\|_\infty^2 }\exp\left( (4d^2+16d^3)t\, \|F_0\|_\infty^2 \right)\\
&+\frac{1}{N_h \ep_h^{1/2}}\sum_{i: i\not=j, i\le N}\int_0^t P\left( |B_i(s)-B_j(s)+\gamma_{i,j}|\le 2\ep_h\right)\,ds. 
\end{aligned}
\eeq
According to (\ref{AS14}) to proof Proposition \ref{Lemma 4.1} it is sufficient to have the following lemma for standard Brownian motions:
\begin{lemma}
\label{Lemma 4.2}
For any $t\ge 0$, there is some constant $C^*_1(t)$ and $C^*_2(t)$ that depends only on $t$ such that
\beq
\frac{1}{N_h}\sum_{i: i\not=j, i\le N_h}\left[\int_0^t P\left( |B_i(s)-B_j(s)+\gamma_{i,j}|\le 2\ep_h\right)ds\right]\le C^*_1(t)\ep_h^{d}+C^*_2(t)\frac{1}{N_h}.
\eeq
\end{lemma}

\begin{proof}
We first note that for any $s$ and $i,j$, $B_i(s)-B_j(s)+\gamma_{i,j}$ has a $d$-dimensional normal distribution with mean $\gamma_{i,j}$ and variance $2s$. So we have
\beq
\label{AS15}
\begin{aligned}
&\int_0^t P\left( |B_i(s)-B_j(s)+\gamma_{i,j}|\le 2\ep_{h}\right)ds\\
&=\int_0^t \int_{  |x|\le 2\ep_{h}}\frac{1}{(4\pi s)^{d/2}}\exp\left( - \frac{|\gamma_{i,j}-x|^2}{4s}\right)\,dx ds\\
&=\int_{  |x|\le 2\ep_{h}}\int_0^t \frac{1}{(4\pi s)^{d/2}}\exp\left( - \frac{|\gamma_{i,j}-x|^2}{4s}\right)\, ds dx.
\end{aligned}
\eeq
To deal with equation (\ref{AS15}), we need to separate the case of $d=1$, $d=2$ and $d\ge 3$.

\mn {\bf Case 1: $d=1$.} In this case we simply use the bound 
$$
\begin{aligned}
&\int_0^t P\left(  |B_i(s)-B_j(s)+\gamma_{i,j}|  \le 2\ep_{h}\right)\,ds\\
&\le\int_{-2\ep_{h}}^{2\ep_{h}}\int_0^t s^{-1/2}\, ds dx=8\ep_{h} \sqrt{t}. 
\end{aligned}
$$
Averaging over $m$ gives us the desired result.

\mn {\bf Case 2: $d=2$.} In this case we have 
$$
\begin{aligned}
&\int_0^t P\left(   |B_i(s)-B_j(s)+\gamma_{i,j}| \le 2\ep_{h}\right)\, ds\\
&=\int_{  |x|\le 2\ep_{h}}\int_0^t \frac{1}{4\pi s}\exp\left( - \frac{|\gamma_{i,j}-x|^2}{4s}\right)\, ds dx.
\end{aligned}
$$
If $\gamma_{i,j}\ge 1$, then for all $\ep_{h}<1/4$ and $x<2\ep_{h}$ we have 
\beq
\label{AS16}
\begin{aligned}
&\int_{  |x|\le 2\ep_{h}}\int_0^t \frac{1}{4\pi s}\exp\left( - \frac{ |\gamma_{i,j}-x|^2}{4s}\right)ds dx\\
&\le \int_{ |x|\le 2\ep_{h}}\int_0^t \frac{1}{s}\exp\left( -\frac1{16s}\right)ds dx\\
&\le 16\ep_{h}^2\int_0^t \frac{1}{s}\exp\left( -\frac1{16s}\right)ds 
\end{aligned}
\eeq
When $\gamma_{i,j}< 1$ taking $h= \frac{ |\gamma_{i,j}-x |^2}{4s}$, we have
$$
\begin{aligned}
&\int_0^t P\left(  |B_i(s)-B_j(s)+\gamma_{i,j}|\le 2\ep_{h}\right)\,ds\\
&=\int_{ |x|\le 2\ep_{h}}\int_0^t \frac{1}{4\pi s}\exp\left( - \frac{|\gamma_{i,j}-x|^2}{4s}\right)\,ds dx\\
&=\int_{ |x|\le 2\ep_{h}}\int_{  |\gamma_{i,j}-x|^2/4t}^\infty \frac1{4\pi h} \exp(-h)\,dh dx. 
\end{aligned}
$$
Note that $h^{-1}\exp(-h)<h^{-1}$ and $h^{-1}\exp(-h)\le \exp(-h)$ when $h\ge 1$. We have
\beq
\label{AS17}
\begin{aligned}
\int_{  |\gamma_{i,j}-x|^2/4t}^\infty h^{-1}\exp(-h)\, dh&\le \int_{  |\gamma_{i,j}-x|^2/4t}^1 h^{-1} dh+\int_1^\infty e^{-h}dh\\
&\le 2|\log( |\gamma_{i,j}-x|)|+|\log t|+1+\log 4.
\end{aligned}
\eeq
Moreover, let $\delta_1=CN_h^{-1/2}$, where $C=L_D^{1/d}$, $\delta_2=\delta_1+4\ep_{h}$ and $M=[\delta_2^{-1}]+1$. For all $k=0,1,\cdots, M$ consider the following sets
\beq
A_k:= \big\{ i: k\delta_2\le |\gamma_{i,j}|<(k+1)\delta_2 \big\}. 
\eeq
By definition, it is easy to see that when $N_h$ is large and $\ep_{h}$ is small
\beq
\bigcup_{k=0}^M A_k\supset \big\{i: |\gamma_{i,j}|<1\big\}.
\eeq
If we first look at $A_0$, according to that $h\ge CN_h^{-1/2}$, the little balls $\{N(\gamma_{i,j},\delta)\}_{i\le N_h, i\not=j}$ (where $N(x,y)$ is the neighborhood of $x$ with radius $y$) have no intersections with each other. And for all $i\in A_0$, 
$$
N(\gamma_{i,j},\delta_1)\subset N(0,\delta_1+\delta_2)
$$
This immediately implies that 
$$
{\rm card} (A_0)\le \left(\frac{\delta_1+\delta_2}{\delta_1}\right)^2=\left( 2+\frac{4\ep_{h}}{C}N_h^{1/2}\right)^2\le 8+\frac{32\ep_{h}^2}{C^2}N_h, 
$$
since the sum of areas of disjoint disks with radius $\delta_1$ in $A_0$ cannot be larger than the area of $A_0$ itself. Thus we have
\beq
\label{AS18}
\frac{1}{N_h}\sum_{i\in A_0} \int_0^t P\left(  |B_i(s)-B_j(s)+\gamma_{i,j}| \le 2\ep_{h}\right)ds\le \frac{t}{N_h} {\rm card} (A_0)\le \frac{8t}{N_h}+\frac{32\ep_{h}^2t}{C^2}
\eeq
Similarly, for each $k\ge 1$ and $i\in A_k$, 
$$
N(\gamma_{i,j},\delta_1) \subset \big\{y: (k-1)\delta_2\le  |y|<(k+2)\delta_2 \big\}
$$
which implies that 
$$
\textrm{card}(A_k)\le \frac{[(k+2)^2-(k-1)^2]\delta_2^2}{\delta_1^2}\le 9k\left(1+\frac{4\ep_{h} N_h^{1/2}}{C}\right)^2. 
$$
Noting that for all $i\in A_k$ and $|x|\le 2\ep_{h}$
$$
|\log(|\gamma_{i,j}-x|)|\le \max\{\log 2, |\log(|k\delta_2-2\ep_{h}|)|\} \le \log 2+|\log(k\delta_2)|. 
$$
Thus according to (\ref{AS17}) and the inequality above
\beq
\label{AS19}
\begin{aligned}
&\frac{1}{N_h}\sum_{i\in A_k} \int_0^t P
\left(  |B_i(s)-B_j(s)+\gamma_{i,j}| \le 2\ep_{h}\right)\,ds \\
&\le \frac{1}{N_h}\sum_{i\in A_k}\left[
  \int_{ |x|\le 2\ep_{h}} |\log t|+\log 4+1+2|\log( |\gamma_{i,j}-x|)| \,dx\right]\\
&\le \frac{1}{N_h}\sum_{i\in A_k}\left[
  \int_{ |x|\le 2\ep_{h}} |\log t|+\log 16+1+ 2|\log(k\delta_2)|\,dx\right]\\
&\le \left[\frac{1}{N_h}\sum_{i\in A_k}16(|\log t|+\log 16+1)\ep_{h}^2\right]+\frac{288\ep_{h}^2 k\left(1+\frac{4\ep_{h} N_h^{1/2}}{C}\right)^2}{N_h}|\log(k\delta_2)|\\
&=\left[\frac{1}{N_h}\sum_{i\in A_k}16(|\log t|+\log 16+1)\ep_{h}^2\right]+\frac{288\ep_{h}^2}{C^2}\delta_2 \left[k\delta_2 |\log(k\delta_2)|\right]
\end{aligned}
\eeq
Summing over $k=0,1,\cdots, M$ we have 
\beq
\begin{aligned}
&\frac{1}{N_h}\sum_{i: \gamma_{i,j}< 1}
 \int_0^t P\left( |B_i(s)-B_j(s)+\gamma_{i,j}| \le 2\ep_{h}\right)ds\\
&\le \frac{8t}{N_h}+\frac{32\ep_{h}^2t}{C^2}+16(|\log t|+\log 16+1)\ep_{h}^2+\frac{288\ep_{h}^2}{C^2}\sum_{k=1}^{[\delta_2^{-1}]+1}\delta_2 \left[k\delta_2 |\log(k\delta_2)|\right]
\end{aligned}
\eeq
Note that the last term in the inequality above is a Riemann sum of function $x|\log x|$ and the fact that $x|\log x|\le \max\{\log 2, e^{-1}\}<1$ on $[0,2]$.
$$
\sum_{k=1}^{[\delta_2^{-1}]+1}\delta_2 \left[k\delta_2 |\log(k\delta_2)|\right]\le\int_0^2 dt=2. 
$$
So we have 
\beq
\label{AS110}
\begin{aligned}
&\frac{1}{N_h}\sum_{i: \gamma_{i,j}< 1}
 \int_0^t P\left( |B_i(s)-B_j(s)+\gamma_{i,j}| \le 2\ep_{h}\right)ds\\
&\le \frac{8t}{N_h}+\frac{32\ep_{h}^2t}{C^2}+16(|\log t|+\log 16+1)\ep_{h}^2+\frac{576\ep_{h}^2}{C^2}
\end{aligned}
\eeq
Combining (\ref{AS110}) and (\ref{AS16}), and letting
\beq
\label{AS111}
\begin{aligned}
&C^*_1(t) :=16\int_0^t \frac{1}{s}\exp\left( -1/16s\right)ds+\frac{32t}{C^2}+16(|\log t|+\log 16+1)+\frac{576}{C^2}\\
&C^*_2(t) :=8t
\end{aligned}
\eeq
we finally get 
$$
\frac{1}{N_h}\sum_{i: i\not=j, i\le N_h}
\int_0^t P\left( |B_i(s)-B_j(s)+\gamma_{i,j}| \le 2\ep_{h}\right)\,ds
\le C^*_1(t)\ep_{h}^{d}+ C^*_2(t)\frac{1}{N_h}
$$
when $d=2$, and the proof for case 2 is complete.

\mn {\bf  Case 3: $d\ge 3$.} The proof in this case is similar but simpler than the case of $d=2$. Again we have 
$$
\begin{aligned}
&\int_0^t P\left( |B_i(s)-B_j(s)+\gamma_{i,j}| \le 2\ep_{h}\right)\,ds \\
&=\int_{|x|\le 2\ep_{h}}
\int_0^t \frac{1}{(4\pi s)^{d/2}}\exp\left( -  \frac{|\gamma_{i,j}-x|^2}{4s} \right)\, ds dx.
\end{aligned}
$$
If $\gamma_{i,j}\ge 1$, then for all $\ep_{h}<1/4$ and $ |x|<2\ep_{h}$ we have 
\beq
\label{AS112}
\begin{aligned}
&\int_{|x|\le 2\ep_{h}}
\int_0^t \frac{1}{(4\pi s)^{d/2}}\exp\left( - \frac{|\gamma_{i,j}-x|^2}{4s}\right)\,ds dx\\
&\le \int_{  |x|\le 2\ep_{h}}\int_0^t \frac{1}{s^{d/2}}\exp\left( -\frac1{16s}\right)\, ds dx\\
&\le 2^{2d}\ep_{h}^d \int_0^t \frac{1}{s^{d/2}}\exp\left( -\frac1{16s}\right)\, ds 
\end{aligned}
\eeq
When $\gamma_{i,j}< 1$ taking $h= \frac{|\gamma_{i,j}-x|^2}{4s}$, we have
\beq
\label{AS113}
\begin{aligned}
&\int_0^t P\left(  |B_i(s)-B_j(s)+\gamma_{i,j}| \le 2\ep_{h}\right)\, ds\\
&=\int_{  |x|\le 2\ep_{h}}\int_0^t \frac{1}{(4\pi s)^{d/2}}\exp\left( - \frac{|\gamma_{i,j}-x|^2}{4s}\right)\, ds dx\\
&<C_d \int_{  |x|\le 2\ep_{h}}|\gamma_{i,j}-x|^{-d+2}dx. 
\end{aligned}
\eeq
where constant 
$$
C_d  := 2^{4d}\int_{0}^\infty h^{-2+d/2}\exp(-h)dh.
$$
Then again we can define $\delta_1=CN_h^{-1/d}$, where $C=L_D^{1/d}$, $\delta_2=\delta_1+4\ep_{h}$ and $M=[\delta_2^{-1}]+1$. For all $k=0,1,\cdots, M$ consider the following sets
\beq
A_k:= \big\{ i: k\delta_2\le |\gamma_{i,j}|<(k+1)\delta_2 \big\}. 
\eeq
such that
$$
\bigcup_{k=0}^M A_k\supset \big\{i: |\gamma_{i,j}|<1\big\}.
$$
Then similarly, we have 
$$
{\rm card} (A_0)\le \left(\frac{\delta_1+\delta_2}{\delta_1}\right)^d=\left( 2+\frac{4\ep_{h}}{C}N_h^{1/d}\right)^d\le 2^{2d-1}+\frac{2^{3d-1}\ep_{h}^d}{C^d}N_h
$$
and
$$
\textrm{card}(A_k)\le \frac{[(k+2)^d-(k-1)^d]\delta_2^d}{\delta_1^d}\le 3^dk^{d-1}\left(1+\frac{4\ep_{h} N_h^{1/d}}{C}\right)^d. 
$$
Thus
\beq
\label{AS114}
\frac{1}{N_h}\sum_{i\in A_0} \int_0^t P\left(  |B_i(s)-B_j(s)+\gamma_{i,j}| \le 2\ep_{h}\right)\,ds\le \frac{t}{N_h}{\rm card} (A_0) \le \frac{2^{2d-1} t}{N_h}+\frac{2^{3d-1}t\ep_{h}^d}{C^d}
\eeq
and
\beq
\begin{aligned}
&\frac{1}{N_h}\sum_{i\in A_k} \int_0^t P\left( |B_i(s)-B_j(s)+\gamma_{i,j}|\le 2\ep_{h}\right)\,ds\\
&\le \frac{C_d}{N_h}\sum_{m\in A_k} \int_{  |x|\le 2\ep_{h}}|\gamma_{i,j}-x|^{-d+2}dx\\
&\le \frac{C_d}{N_h} 3^dk^{d-1}\left(1+\frac{4\ep_{h} N_h^{1/d}}{C}\right)^d (2^{2d} \ep_{h}^d)\times \left(2^d [k(CN_h^{-1/d}+2\ep_{h})]^{-d+2}\right)
\end{aligned}
\eeq
 Summing over $k=0,1,\cdots,M$, 
\beq
\begin{aligned}
&\frac{1}{N_h}\sum_{i: \gamma_{i,j}< 1} 
\int_0^t P\left( |B_i(s)-B_j(s)+\gamma_{i,j}|\le 2\ep_{h}\right)\,ds\\
&\le \frac{2^{2d-1} t}{N_h}+\frac{2^{3d-1}t\ep_{h}^d}{C^d}+C_d \left( \frac{24}{C}\right)^d \ep_{h}^d \sum_{k=1}^{[\delta_2^{-1}]+1}[(k\delta_2)\delta_2].
\end{aligned}
\eeq
Again for the last term we have 
$$
\sum_{k=1}^{[\delta_2^{-1}]+1}[(k\delta_2)\delta_2]\le \int_0^2 t dt=2. 
$$
Thus
\beq
\label{AS115}
\begin{aligned}
&\frac{1}{N_h}\sum_{i: \gamma_{i,j}< 1} 
\int_0^t P\left( |B_i(s)-B_j(s)+\gamma_{i,j}|\le 2\ep_{h}\right)\,ds\\
&\le \frac{2^{2d-1} t}{N_h}+\frac{2^{3d-1}t\ep_{h}^d}{C^d}+2C_d \left( \frac{24}{C}\right)^d \ep_{h}^d
\end{aligned}
\eeq
Then combining (\ref{AS112}) and (\ref{AS115}), and letting
\beq
\label{AS116}
\begin{aligned}
&C^*_1(t)  :=2^{2d} \int_0^t \frac{1}{s^{d/2}}\exp\left( -\frac1{16s}\right)\, ds+ \frac{2^{3d-1}t}{C^d}+2C_d \left( \frac{24}{C}\right)^d \\
&C^*_2(t)  :=2^{2d-1}t
\end{aligned}
\eeq
We complete the proof of case 3.
\end{proof}
\mn With Lemma \ref{Lemma 4.2} proved, then according to (\ref{AS14}), let 
\beq
\label{AS117}
\begin{aligned}
C_1(t)&= \frac{1}{(4d^2+16d^3)\, \|F_0\|_\infty^2 }\exp\left( (4d^2+16d^3)t\, \|F_0\|_\infty^2 \right)+C^*_1(t)\\
C_2(t)&=C^*_2(t). 
\end{aligned}
\eeq
Then the proof of Proposition \ref{Lemma 4.1} is complete. 
\end{proof}

\section{Estimation of the Truncation Error}
Back on estimating the errors times, we will first estimate the term $Tr(s)$ of the truncation error and have the proposition as follows: 

\begin{prop}
\label{truncate}
For $Tr(s)$ be the truncation error which is defined as 
\begin{align*}
\label{Prop14}
Tr(s)
=&2 \int_0^s \int_{{\mathbb R}^d} 
\left[\sum_{i=1}^{N_h} h^d\rho_0(\theta_{i,h}h) \varphi_{\ep_h}\left( x-\hat X_{h,i}(q)\right)
\left( F(x,q) - F\left(\hat X_{h,i}(q),q\right)  \right) \right]
\cdot \nabla  (\rho -\hat \rho_{h})(x,q) \, dx dq\\
&+h^{2d}s \|\nabla \varphi_{\ep_h}\|^2 \sum_{i=1}^{N_h} \rho_0(\theta_{i,h}h)^2
\end{align*}
in Proposition 1. Then we have when $h$ is sufficiently small, 
\beq
\label{TR1}
\begin{aligned}
E&\left(\max\left\{ \sup_{s\le t}\left\{Tr(s)-\frac{1}{2} \int_0^s\|\nabla (\rho-\hat\rho_{h})(\cdot,q) \|^2\, dq\right\}, 0\right\}\right)\\
&\le h^dt\ep_h^{-d-2}\|\nabla \varphi\|^2\|\rho_0\|_\infty U_D+2U_D^2 L_F^2\|\varphi\|^2 \|\rho_0\|_\infty^2 \left[\frac{t}{N_h\ep_h^{d-2}}+C_1(t)\ep_h+C_2(t)\frac{1}{N_h\ep_h^{d-1}}\right]
\end{aligned}
 \eeq
 where $C_1$ and $C_2$ are the constants in Proposition \ref{Lemma 4.1}. 
\end{prop}

\begin{proof}
First for the constant term, noting that $\rho_0$ is a bounded function and that 
$$
\|\nabla \varphi_{\ep_h}\|^2 =\ep_h^{-2d-2}\int_{\mathbb{R}^d} \left|\nabla\varphi\left(\frac{x}{\ep_h}\right)\right|^2 dx=\ep_h^{-d-2}\|\nabla \varphi\|^2,
$$
and that $\ep_h=h^{1/4d}$, we have for any $s\in [0,t]$
\beq
\label{constant}
\begin{aligned}
sh^{2d} \|\nabla \varphi_{\ep_h}\|^2 \sum_{i=1}^{N_h} \rho_0(\theta_{i,h}h)^2&\le th^d\ep_h^{-d-2}\|\nabla \varphi\|^2\|\rho_0\|_\infty U_D
 \end{aligned}
\eeq
when $h$ is sufficiently small. Thus, we will concentrate on the non-constant part in the truncation error. By Cauchy Schwarz inequality, 
\begin{align*}
2 \int_0^s& \int_{{\mathbb R}^d} 
\left[\sum_{i=1}^{N_h} h^d\rho_0(\theta_{i,h}h) \varphi_{\ep_h}\left( x-\hat X_{h,i}(q)\right)
\left( F(x,q) - F\left(\hat X_{h,i}(q),q\right)  \right) \right]
\cdot \nabla  (\rho -\hat \rho_{h})(x,q) \, dx dq\\
&\le2 \int_0^s \int_{{\mathbb R}^d} 
\left|\sum_{i=1}^{N_h} h^d\rho_0(\theta_{i,h}h) \varphi_{\ep_h}\left( x-\hat X_{h,i}(q)\right)
\left( F(x,q) - F\left(\hat X_{h,i}(q),q\right)  \right) \right|^2 dx dq\\
&+\frac{1}{2} \int_0^s\|\nabla (\rho-\hat\rho_{h})(\cdot,q)\|^2 dq.
\end{align*}
Let 
\beq
\label{TR2}
Res(s)=\int_0^s \int_{{\mathbb R}^d} 
\left|\sum_{i=1}^{N_h} h^d\rho_0(\theta_{i,h}h) \varphi_{\ep_h}\left( x-\hat X_{h,i}(q)\right)
\left( F(x,q) - F\left(\hat X_{h,i}(q),q\right)  \right) \right|^2 dx dq.
\eeq
Then by definition in order to show Proposition 2 it is sufficient to prove that 
\beq
\label{Res1}
E\left(\sup_{s\le t} Res(s)\right)\le U_D^2 L_F^2\|\varphi\|^2 \|\rho_0\|_\infty^2 \left[\frac{t}{N_h\ep_h^{d-2}}+C_1(t)\ep_h+C_2(t)\frac{1}{N_h\ep_h^{d-1}}\right]
\eeq
when $h$ is sufficiently small.  To show this, first It is easy to see that we can rewrite the integrand of $Res(s)$ as 
$$
\sum_{i,j=1,2,\cdots,N_h} h^{2d}\rho_0(\theta_{i,h}h)\rho_0(\theta_{j,h}h) R_{i,j}(x,q)
$$
where
\begin{align*}
R_{i,j}(x,q)=\varphi_{\ep_h}\left( x-\hat X_{h,i}(q)\right)
\left( F(x,q) - F\left(\hat X_{h,i}(q),q\right) \right)\cdot \varphi_{\ep_h}\left( x-\hat X_{h,j}(q)\right)
\left( F(x,q) - F\left(\hat X_{h,j}(q),q\right) \right)
\end{align*}
Note that for any $i,j\le N$, $R_{i,j}(x,q)\equiv0$ when $|X_j(q)-X_i(q)|>2\ep_h$. And when $|X_i(q)-X_j(q)| \le 2\ep_h$, noting that $F$ is Lipschitz continuous with the Lipschitz constant less than or equal to $L_F$, 
$$
\left| R_{i,j}(x,q)\right|\le L_F^2\ep_h^2 \left| \varphi_{\ep_h}\left( x-\hat X_{h,i}(q)\right) \varphi_{\ep_h}\left( x-\hat X_{h,j}(q)\right)\right|
$$
Thus for all $i,j\le N_h$, we have the spatial integral 
\beq
\label{Res2}
\int_{{\mathbb R}^d} |R_{i,j}(x,q)| \, dx\le \ep^{2-d} L_F^2\|\varphi\|^2\, \mathbb{1}_{|\hat X_{h,i}(q)-\hat X_{h,j}(q)|\le 2\ep_h}.
\eeq 
Thus we have for any $s\in [0,t]$, 
\beq
\label{Res3}
Res(s)\le \mathbb{R}^*(s)=h^{2d} \ep_h^{2-d} L_F^2\|\varphi\|^2\sum_{i,j=1,2,\cdots,N_h}\rho_0(\theta_{i,h}h)\rho_0(\theta_{j,h}h) \int_0^s \mathbb{1}_{|\hat X_{h,i}(q)-\hat X_{h,j}(q)|\le 2\ep_h} dq
\eeq
and note that $\mathbb{R}^*(s)$ is monotonically increasing over $s$. Thus to prove Proposition 2, it suffices to show that 
$$
E(\mathbb{R}^*(t))\le U_D^2 L_F^2\|\varphi\|^2 \|\rho_0\|_\infty^2 \left[\frac{t}{N_h\ep_h^{d-2}}+C_1(t)\ep_h+C_2(t)\frac{1}{N_h\ep_h^{d-1}}\right]
$$
when $h$ is sufficiently small. Noting that $\hat X_{h,i}(s)-\hat X_{h,j}(s)$ is continuous and adaptable to $\mathcal{F}^{N}_t$ (which implies progressive), $\mathbb{1}_{|\hat X_{h,i}(s)-\hat X_{h,j}(s)|\le 2\ep_h}\times \mathbb{1}_{0\le s\le t}$ is measurable on $[0,t]\times \Omega$ and bounded and thus integrable. By Fubini's Theorem, 
\beq
\label{R*1}
\begin{aligned}
E[\mathbb{R}^*(t)]=&h^{2d} \ep_h^{2-d} L_F^2\|\varphi\|^2\left(\sum_{i=1:N_h}\rho_0(\theta_{i,h}h)^2t\right.\\
+ &\left.\sum_{i\not=j=1,2,\cdots,N_h} \rho_0(\theta_{i,h}h)\rho_0(\theta_{j,h}h) \int_0^tP\left(|\hat X_{h,i}(s)-\hat X_{h,j}(s)|\le 2\ep_h\right) ds\right)\\
\le& h^{d} U_D\ep_h^{2-d} L_F^2\|\varphi\|^2 \|\rho_0\|_\infty^2 \left[t+\sum_{j=1}^{N_h} E_j(t)\right]
\end{aligned}
\eeq
where for any $j=1,2,\cdots,N_h$, $E_j(t)$ is the separation term defined in \eqref{R*2} in Section 4. With Proposition \ref{Lemma 4.1} proved, then combining \eqref{lemma 4.1}, \eqref{Res1}, \eqref{Res3}, and \eqref{R*1}  we have the inequality in Proposition \ref{truncate}. 
\end{proof}

\section{Estimation of the Martingale Error}
In this section, we estimate the martingale error $\bar M_s=M_s+\tilde M_s$. Out first result is about $M_s$: 
\begin{lemma}\label{L1}
For all $s\in [0,t]$, we have the second moment control 
\beq
\label{L1.1}
E\big(M_s^2 \big)\le \frac{4h^{d}U_D}{\ep^{d+2}}\|\nabla\varphi\|^2\|\rho_0\|^2_\infty\int_0^s\|\rho(\cdot,q)\|^2\, dq.
\eeq
\end{lemma} 
\begin{proof}
Here and in Lemma \ref{L2}, we will use the natural filtration $\mathcal{F}^{N_h}_s$, which is generated by the Brownian motions $B_1(s),\cdots B_{N_h}(s)$. Note that $M_s=\sum_{i=1}^{N_h} \rho_0(\theta(i,h)h)M_s^i$ where 
\begin{align*}
M_s^i&= 2h^d\int_0^s \int_{{\mathbb R}^d}\rho(x,q) \nabla \varphi_{\ep_h}\left(x-\hat X_{h,i}(s)\right)dx\cdot d B_i(q)=\sum_{k=1}^d  M_s^{i,k},
\end{align*}
and
$$
M_s^{i,k}=2h^d \int_0^s \int_{{\mathbb R}^d}\rho(x,q) \frac{\partial \varphi_{\ep_h}\left(x-\hat X_{h,i}(q)\right)}{\partial x_k}dx ~ d B^{(k)}_i(q).
$$
The $B^{(k)}_i(s)$ in the equation above is the $k$th coordinate of the Brownian motion $B_i(s)$ and it is itself a one dimension Brownian motion and a square integrable martingale under filtration $\mathcal{F}^{N_h}_s$ noting that $B^{(k)}_i(s)$ is independent to $B^{(h)}_j(s)$ for all $h\not=k$, or $i\not=j$. For each $i$ and $k$ we have the integrand 
$$
Y_{i,k}(q)=\int_{{\mathbb R}^d}\rho(x,q) \frac{\partial \varphi_{\ep_h}\left(x-\hat X_{h,i}(q)\right)}{\partial x_k}dx
$$
continuous and adapted to filtration $\mathcal{F}^{N}_q$. Moreover 
$$
|Y_{i,k}(q)|= \left|\int_{{\mathbb R}^d}\rho(x,q) \frac{\partial \varphi_{\ep_h}\left(x-\hat X_{h,i}(q)\right)}{\partial x_k}dx\right|  
\le \left\|  \frac{\partial \varphi_{\ep_h}}{\partial x_k}  \right\| \times\|\rho(\cdot,s)\| <\infty.
$$
Thus by Theorem 5.2.3 in \cite{EK}, for all $i\in \{1, 2,\cdots, N\}$ and $k\in\{1,2,\cdots, d\}$, $M_s^{i,k}$ is a square integrable martingale with 
\beq
\label{L2IK}
E[(M_s^{i,k})^2]=4h^{2d} E\left( \int_0^s Y_{i,k}(q)^2 dq\right)\le 4h^{2d}\left\|  \frac{\partial \varphi_{\ep_h}}{\partial x_k}  \right\|^2 \int_0^t \|\rho(\cdot,s)\|^2ds.
\eeq
And for all $(i,k)\not=(j,h)$ we have that 
\beq
\label{QCov1}
\begin{aligned}
\langle M_s^{i,k},M_s^{j,h}\rangle&=\langle Y_{i,k}\cdot B_i^{(k)}(s), Y_{i,h}\cdot B_j^{(h)}(s)\rangle\\
&=\int_0^s Y_{i,k}(s)\cdot Y_{j,h}(q) \, d\langle B_i^{(k)}(q), B_j^{(h)}(q)\rangle\\
&=\int_0^s Y_{i,k}(s)\cdot Y_{j,h}(q) \, d0=0,
\end{aligned}
\eeq
since $\langle B_i^{(k)}(s), B_j^{(h)}(s)\rangle\equiv 0$ for two independent Brownian motions, where $\langle X_s,Y_s\rangle$ is the quadratic covariance between the two processes $X_s$ and $Y_s$, defined by
$$
\langle X_s,Y_s\rangle=\frac{1}{2}(\langle X_s+Y_s\rangle-\langle X_s\rangle-\langle Y_s\rangle).
$$
Noting that $M_s^{i,k}$ and $M_s^{j,h}$ are both square integrable martingales, (\ref{QCov1}) implies that 
\beq
\label{MeanMix}
E\left( M_s^{i,k} M_s^{j,h}\right)\equiv 0.
\eeq 
Combining (\ref{L2IK}) and (\ref{MeanMix}) immediately gives us 
$$
E[(M_s^i)^2]=\sum_{k=1}^d E[(M_s^{i,k})^2] \le 4h^{2d} \|\nabla \varphi_{\ep_h}\|^2 \int_0^s\|\rho(\cdot,q)\|^2\, dq. 
$$ 
and
$$
E(M_s^i M_s^j)=0
$$
which implies that 

\beq
E\big( (M_s)^2\big)=\sum_{i=1}^{N_h} \rho_0(\theta(i,h)h)^2E[(M_t^i)^2]\le \frac{4h^{d}U_D}{\ep^{d+2}}\|\nabla\varphi\|^2\|\rho_0\|^2_\infty\int_0^s\|\rho(\cdot,q)\|^2\, dq.
\eeq
\end{proof}

\mn Then we estimate the second part of the martingale error and have a lemma as follows: 
\begin{lemma}\label{L2}
For all $s\in [0,t]$, we have the second moment control 
\beq
\label{L 1.2}
E\big( (\tilde M_s)^2 \big)\le \frac{4h^{d}U_D^3}{\ep^{2d+2}}\|\varphi\|^2\|\nabla \varphi\|^2\|\rho_0\|_\infty^4 s.
\eeq
\end{lemma}

\begin{proof}
Again note that $\tilde M_s=\sum_{i=1}^{N_h}\rho_0(\theta(i,h)h)\tilde M_s^i$ with 
$$
\tilde M_s^i=\sum_{k=1}^d \tilde M_s^{i,k}
$$
where 
$$
\tilde M_s^{i,k}=\int_0^s Z_{i,k}(q)dB^{(k)}_i(q)
$$
and
\beq
\begin{aligned}
Z_{i,k}(q)= &2h^{2d} \int_{{\mathbb R}^d}\varphi_{\ep_h}(x) \sum_{j=1}^{i-1} \rho_0(\theta(j,h)h)\frac{\partial\varphi_{\ep_h}\left(x+\hat X_{h,i}(q)-\hat X_{h,j}(q)\right)}{\partial x_k} dx\\
-&2h^{2d} \int_{{\mathbb R}^d}\varphi_{\ep_h}(x) \sum_{j=i+1}^{N_h} \rho_0(\theta(j,h)h)\frac{\partial\varphi_{\ep_h}\left(x+\hat X_{h,i}(q)-\hat X_{h,j}(q)\right)}{\partial x_k} dx.
\end{aligned}
\eeq
It is easy to see that the integrand $Z_{i,k}(q)$ is continuous and adapted to $\mathcal{F}^{N_h}_q$ and that 
\beq
|Z_{i,k}(q)|\le 2h^{d}U_D\|\varphi_{\ep_h}\|\cdot \left\|\frac{\partial \varphi_{\ep_h}}{\partial x_k}\right\|\cdot \|\rho_0\|_\infty
\eeq
Then again accordion to Theorem 5.2.3 in \cite{EK} we have for all $i\in \{1, 2,\cdots, N_h\}$ and $k\in\{1,2,\cdots, d\}$, $M_s^{i,k}$ is a square integrable martingale with that 
\beq
\label{L2IK2}
E[(\tilde M_s^{i,k})^2]= E\left(\int_0^s Z_{i,k}(q)^2 dq\right)\le 4h^{2d}U_D^2\|\varphi_{\ep_h}\|^2\cdot \left\|\frac{\partial \varphi_{\ep_h}}{\partial x_k}\right\|^2\cdot \|\rho_0\|_\infty^2 s. 
\eeq
and that for all $(i,k)\not=(j,h)$ we have that 
\beq
\label{QCov}
\begin{aligned}
\langle \tilde M_s^{i,k},\tilde M_s^{j,h}\rangle&=\langle Z_{i,k}\cdot B_i^{(k)}(s), Z_{i,h}\cdot B_j^{(h)}(s)\rangle\\
&=\int_0^s Z_{i,k}(q)\cdot Z_{j,h}(q) \, d\langle B_i^{(k)}(q), B_j^{(h)}(q)\rangle\\
&=\int_0^s Z_{i,k}(q)\cdot Z_{j,h}(q) \, d0=0,
\end{aligned}
\eeq
which implies that 
\beq
\label{MeanMix2}
E\left( \tilde M_s^{i,k} \tilde M_s^{j,h}\right)\equiv 0.
\eeq 
Thus we immediately have 
$$
E[(\tilde M_s^{i})^2]\le 4h^{2d}U_D^2\|\varphi_{\ep_h}\|^2\cdot \left\|\nabla \varphi_{\ep_h}\right\|^2\cdot \|\rho_0\|_\infty^2 s
$$
and $E \big(\tilde M_s^i \, \tilde M_s^j\big) =0$ for all $i\not=j$. Thus 
\beq
E\big(  (\tilde M_s)^2\big)
=\sum_{i=1}^{N_h} \rho_0(\theta(i,h)h)^2 E \big( (\tilde M_s^i)^2 \big) 
\le  \frac{4h^{d}U_D^3}{\ep^{2d+2}}\|\varphi\|^2\|\nabla \varphi\|^2\|\rho_0\|_\infty^4 s.
\eeq
\end{proof}

\section{Estimation of the Initial Error}
\mn To estimate the distance between the initial empirical density 
$$
\rho_{h}(x,0)=\sum_{i=1}^{N_h} h^d \rho_0(\theta_{i,h}) \varphi_{\ep_h}\left(x-\theta_{i,h}\right). 
$$
and $\rho_0(x)$, we have a lemma as follows:
\begin{lemma}
\label{initial lemma}
For Sufficiently small $h$, 
$$
\big\|\rho_{h}(x,0)-\rho_0(x)\big\|^2\le \ep_h^{1/2}. 
$$
\end{lemma}

\begin{proof}
To prove this lemma, we introduce the following intermediate density function 
\beq
\label{initial intermediate}
\tilde \rho_{h}(x,0)=\rho_0(x)\left[\sum_{i=1}^{N_h} h^d \varphi_{\ep_h}\left(x-\theta_{i,h}\right)\right]. 
\eeq
To estimate the distance between $\tilde \rho_{h}(x,0)$ and $\rho_h(x,0)$, we have for any $x\in D_1$, 
\begin{align*}
\big|\rho_{h}(x,0)-\tilde \rho_{h}(x,0)\big|&\le \sum_{i=1}^{N_h} h^d |\rho_0(\theta_{i,h})-\rho_0(x)| \varphi_{\ep_h}\left(x-\theta_{i,h}\right)\\
&=\sum_{|x-\theta_{i,h}|_\infty\le \ep_h} h^d |\rho_0(\theta_{i,h})-\rho_0(x)| \varphi_{\ep_h}\left(x-\theta_{i,h}\right)
\end{align*}
Noting that $\rho_0$ is a Lipschitz continuous function with Lipschitz constant $L_{\rho_0}$, and that 
$$
\varphi_{\ep_h}(x)=\frac{1}{\ep_h^d} \varphi\left( \frac{x}{\ep_h}\right)
$$ 
we have 
\begin{align*}
\big|\rho_{h}(x,0)-\tilde \rho_{h}(x,0)\big|&\le\sum_{|x-\theta_{i,h}|\le \ep_h} h^d |\rho_0(\theta_{i,h})-\rho_0(x)| \varphi_{\ep_h}\left(x-\theta_{i,h}\right)\\
&\le L_{\rho_0}\|\varphi\|_\infty\sum_{|x-\theta_{i,h}|_\infty\le \ep_h} \frac{h^d}{\ep_h^{d-1}}. 
\end{align*}
Recalling by definition $\ep_h=h^{1/6d}\gg h$, when $h$ is sufficiently small, we have 
\beq
\label{initial 1}
\big|\rho_{h}(x,0)-\tilde \rho_{h}(x,0)\big|\le L_{\rho_0}\|\varphi\|_\infty 3^d \ep_h.
\eeq
And for any $x\in D_1^c$
$$
\tilde \rho_{h}(x,0)=\rho_h(x,0)=0.
$$
And for the distance between $\tilde \rho_{h}(x,0)$ and $\rho_0(x)$, we first note that for any $x\in D$
$$
\tilde \rho_{h}(x,0)=\rho_0(x)\left[\sum_{i\le N_h, |x-\theta_{i,h}|\le \ep_h} h^d \varphi_{\ep_h}\left(x-\theta_{i,h}\right)\right].
$$
Then for any $x\in D$ and $i$ such that $|x-\theta_{i,h}|\le \ep_h$, 
$$
\left|h^d \varphi_{\ep_h}\left(x-\theta_{i,h}\right)-\int_{C(x-\theta_{i,h},h)}\varphi_{\ep_h}(y)dy\right|\le h^{d+1} \ep_h^{-d-1} \|\nabla \varphi\|_\infty.
$$
Summing up over all such neighborhood of size $h$ centered at $x-\theta_{i,h}$, we have 
\begin{align*}
\tilde \rho_{h}(x,0)&\ge \rho_0(x)\left[\int_{|x-y|_\infty\le \ep_h\cap \{y\in D\}}\varphi_{\ep_h}(x-y)dy-\int_{|x-y|_\infty\le 2\ep_h} h \ep_h^{-d-1} \|\nabla \varphi\|_\infty dy\right]\\
&= \rho_0(x)\left[\int_{|x-y|_\infty\le \ep_h\cap \{y\in D\}}\varphi_{\ep_h}(x-y)dy-4^dh\ep_h^{-1}  \|\nabla \varphi\|_\infty \right], 
\end{align*}
and 
\begin{align*}
\tilde \rho_{h}(x,0)&\le \rho_0(x)\left[\int_{|x-y|_\infty\le 2\ep_h}\varphi_{\ep_h}(x-y)dy+\int_{|x-y|_\infty\le 2\ep_h} h \ep_h^{-d-1} \|\nabla \varphi\|_\infty dy\right]\\
&= \rho_0(x)\left[1+4^dh\ep_h^{-1}  \|\nabla \varphi\|_\infty \right], 
\end{align*}
Thus for any $x$ such that $d(x,D^c)> \ep_h$, noting that $|x-y|_\infty\le \ep_h\subset \{y\in D\}$, we have 
\beq
\label{initial 21}
\big|\tilde \rho_{h}(x,0)-\rho_0(x)\big|\le 4^dh\ep_h^{-1}  \|\nabla \varphi\|_\infty\|\rho_0\|_\infty =4^d\ep_h^{5d-1}  \|\nabla \varphi\|_\infty\|\rho_0\|_\infty
\eeq
And for any  $x$ such that $d(x,D^c)\le \ep_h$, we have 
$$
0\le \rho_0(x)\le L_{\rho_0}\ep_h,
$$
and 
$$
0\le \tilde \rho_{h}(x,0)\le \rho_0(x)\left[1+4^dh\ep_h^{-1}  \|\nabla \varphi\|_\infty \right].
$$
Thus
\beq
\label{initial 22}
\big|\tilde \rho_{h}(x,0)-\rho_0(x)\big|\le L_{\rho_0}\ep_h\left[1+4^dh\ep_h^{-1}  \|\nabla \varphi\|_\infty \right]
\eeq
Combining \eqref{initial 21} and \eqref{initial 22}, we have that when $h$ is sufficiently small, for any $x\in D$
\beq
\label{initial 2}
\big|\tilde \rho_{h}(x,0)-\rho_0(x)\big|\le 2L_{\rho_0}\ep_h.
\eeq
And for any $x\in D^c$
$$
\tilde \rho_{h}(x,0)=\rho_0(x)=0.
$$
Thus combining \eqref{initial 1} and \eqref{initial 2}, we have for any $x\in D_1$, 
\beq
\label{initial}
\big|\rho_{h}(x,0)-\rho_0(x)\big|\le L_{\rho_0}\big(2+3^d\|\varphi\|_\infty \big)\ep_h.
\eeq
and for any $x\in D_1^c$
$$
\rho_{h}(x,0)=\rho_0(x)=0.
$$
Thus 
\begin{align*}
\big\|\rho_{h}(x,0)-\rho_0(x)\big\|^2&=\int_{D_1}\big|\rho_{h}(x,0)-\rho_0(x)\big|^2 dx\\
&\le U_D L_{\rho_0}^2\big(2+3^d\|\varphi\|_\infty \big)^2 \ep_h^2\le \ep_h^{1/2}.
\end{align*}
The Proof of Lemma \ref{initial lemma} is complete. 
\end{proof}

\section{Proof of Theorem 2 and Theorem 1}
In this section, we will put all the estimations we have on different error terms together to finish the proof of Theorem 2 and then Theorem 1. First since that in \eqref{TR1} we have 
$$
\begin{aligned}
E&\left(\max\left\{ \sup_{s\le t}\left\{Tr(s)-\frac{1}{2} \int_0^s\|\nabla (\rho-\hat\rho_{h})(\cdot,q) \|^2\, dq\right\}, 0\right\}\right)\\
&\le h^dt\ep_h^{-d-2}\|\nabla \varphi\|^2\|\rho_0\|_\infty U_D+2U_D^2 L_F^2\|\varphi\|^2 \|\rho_0\|_\infty^2 \left[\frac{t}{N_h\ep_h^{d-2}}+C_1(t)\ep_h+C_2(t)\frac{1}{N_h\ep_h^{d-1}}\right]\\
&\le C_{truncate}^2(t)\ep_h
\end{aligned}
$$
since that $\ep_h=h^{1/6d}$, where 
$$
C_{truncate}(t)=\sqrt{\|\nabla \varphi\|^2\|\rho_0\|_\infty U_D t+2U_D^2 L_F^2\|\varphi\|^2 \|\rho_0\|_\infty^2 [L_D t+C_1(t)+C_2(t)]}. 
$$
Thus according to Chebyshev inequality, let event 
$$
A_T=\left\{ \sup_{s\le t}\left\{Tr(s)-\frac{1}{2} \int_0^s\|\nabla (\rho-\hat\rho_{h})(\cdot,q) \|^2\, dq\right\}\le C_{truncate}(t)\ep_h^{1/2} \right\}.
$$
we have $P(A_T)\ge 1- C_{truncate}(t)\ep_h^{1/2}$. 

\mn Similarly, in Lemma \ref{L1} we have that $M_s$ is a $L^2$ integrable martingale and for any $s\le t$, 
$$
E\big(M_s^2 \big)\le \frac{4h^{d}U_D}{\ep^{d+2}}\|\nabla\varphi\|^2\|\rho_0\|^2_\infty\int_0^s\|\rho(\cdot,q)\|^2\, dq.
$$
Since $F$ is Lipschitz continuous from $\mathbb{R}^d\to \mathbb{R}^d$, it is also differentiable almost everywhere by Rademacher's theorem, see Theorem 3.1.6 of \cite{Geometric measure}. Thus for any $x$ such that $F$ is differentiable, we have $|| \nabla\cdot \vec F||_{L^\infty}\le d L_F$, which, according to \eqref{DT}, implies that $|| \rho(\cdot, s)||^2 \le e^{C_0\, s} || \rho_0||^2$ where $C_0 =2d L_F \ge 2 || \nabla\cdot \vec F||_{L^\infty}$. So we have 
$$
E\big(M_t^2 \big)\le \frac{4h^{d}U_De^{C_0t}}{\ep_h^{d+2}C_0}\|\nabla\varphi\|^2\|\rho_0\|^2_\infty\le \frac{4U_De^{C_0t}}{C_0}\|\nabla\varphi\|^2\|\rho_0\|^2_\infty\ep_h^3
$$
since $\ep_h=h^{1/6d}$. By Doob and Chebyshev inequality, letting
$$
C_M(t)=\left(\frac{16U_De^{C_0t}}{C_0}\|\nabla\varphi\|^2\|\rho_0\|^2_\infty \right)^{1/3}
$$
and event
$$
A_M=\left\{\sup_{s\in [0,t]}|M_s|\le C_M(t)\ep_h\right\}
$$
we have $P(A_M)\ge1-C_M(t)\ep_h$. Similarly, according to Lemma \ref{L2}, $\tilde M_s$ is a $L^2$ integrable martingale and
$$
E\big( (\tilde M_t)^2 \big)\le \frac{4h^{d}U_D^3}{\ep_h^{2d+2}}\|\varphi\|^2\|\nabla \varphi\|^2\|\rho_0\|_\infty^4 t\le 4U_D^3\|\varphi\|^2\|\nabla \varphi\|^2\|\rho_0\|_\infty^4 t \ep_h^2
$$
for $\ep_h=h^{1/6d}$. Then let 
$$
C_{\tilde M}(t)=\left(16U_D^3\|\varphi\|^2\|\nabla \varphi\|^2\|\rho_0\|_\infty^4 t \right)^{1/3}
$$
and event 
$$
A_{\tilde M}=\left\{\sup_{s\in [0,t]}|\tilde M_s|\le C_{\tilde M}(t)\ep_h^{2/3}\right\}.
$$
We have again by Doob and Chebyshev inequality, $P(A_{\tilde M})\ge1-C_{\tilde M}(t)\ep_h^{2/3}$. Finally note that in Lemma \ref{initial lemma}, the initial error is bounded by 
$$
\|\rho_0-\hat \rho_{h}(\cdot, 0)\|\le \ep_h^{1/2}
$$
when $h$ is sufficiently small. Then under the event 
$$
A=A_T\cap A_M\cap A_{\tilde M}
$$
such that 
\begin{align*}
P(A)&\ge 1- C_{truncate}(t)\ep_h^{1/2}-C_M(t)\ep_h-C_{\tilde M}(t)\ep_h^{2/3}\\
&\ge 1- [C_{truncate}(t)+C_M(t)+C_{\tilde M}(t)] \ep_h^{1/2}, 
\end{align*}
we have for any $s\in [0,t]$, 
\begin{align*}
\| (\rho-\hat\rho_{h})&(\cdot,s)\|^2+\frac{1}{2}\int_0^ s\|\nabla(\rho-\hat\rho_{h})(\cdot,q)\|^2\, dq \\
&\le \ep_h^{1/2}-\int_0^s \int_{\mathbb{R}^d} \nabla \cdot F(x,q) [\rho(x,q)-\hat \rho_h(x,q)]^2 dx dq+[C_{truncate}(t)+C_M(t)+C_{\tilde M}(t)] \ep_h^{1/2}. 
\end{align*}
Noting that the inequality above holds for all $s\in [0,t]$ and that $\ep_h^{1/2}=h^{1/12d}$, then let 
\beq
\label{FunC1}
c_1(t)=2e^{C_0 t} [1+C_{truncate}(t)+C_M(t)+C_{\tilde M}(t)].
\eeq
Gronwall's inequality finishes the proof of Theorem 2. 

\mn With Theorem 2 proved, combining it with the result of Theorem 3 and let 
\beq
\label{FunC}
c(t)=c_0(t)+c_1(t).
\eeq
The proof of Theorem 1 is complete.

\end{document}